\documentclass[a4paper,12pt]{article}

\usepackage{epsfig}
\usepackage{amsfonts,eucal}
\usepackage{amssymb,amsmath,amsthm,color}
\usepackage{authblk}

\usepackage[T2A]{fontenc}
\usepackage[cp1251]{inputenc}
\newtheorem{remark}{Remark}[section]
\newtheorem{theorem}{Theorem}[section]
\newtheorem{proposition}{Proposition}[section]
\newtheorem{lemma}{Lemma}[section]

\newcommand{\R}{{\mathbb R}}

\newcommand{\X}{{\R^d}}

\newcommand{\eps}{\varepsilon}

\begin{document}

\title{Homogenization of non-autonomous evolution problems for convolution type operators in random media.}

\setcounter{footnote}3


\author[1,2,*]{A. Piatnitski 
}
\author[2,1,$\star$]{E. Zhizhina
}
\affil[1]{\small The Arctic University of Norway, Campus Narvik, Norway}
\affil[2]{\small Institute for Information Transmission Problems of Russian Academy of Siences (IITP RAS), Moscow, Russia}
\affil[*]{ email: {\tt apiatnitski@gmail.com} }
\affil[$\star$]{email: {\tt ejj@iitp.ru}}
\date{}

\maketitle

\begin{abstract}

We study homogenization problem for non-autonomous  parabolic equations of the form $\partial_t u=L(t)u$ with an integral convolution type operator $L(t)$ that has
 a non-symmetric jump kernel which is periodic in spatial variables and stationary random in time.
We show that the  homogenized equation is  a SPDE with a finite dimensional multiplicative noise.
\end{abstract}

\section{Introduction and previous results}\label{Intro}

This paper deals with  homogenization problem for a parabolic type equation of the form
\begin{equation}\label{eq_eps}
\frac{\partial u^\varepsilon}{\partial t} \ = \ \frac{1}{\varepsilon^{d+2}} \int\limits_{\X} a\Big(\frac{x-y}{\eps}\Big) \ \mu_\omega \Big(\frac{x}{\eps}, \frac{y}{\eps}; \frac{t}{\varepsilon^2} \Big)\ (u^\varepsilon(y,t) - u^\varepsilon(x,t)) dy,
\end{equation}
here $\varepsilon>0$ is a small parameter, the coefficient $ \mu_\omega(\xi, \eta; \tau)$ is periodic in spatial variables $\xi$ and $\eta$, and stationary random in time, and the functions $\mu_\omega(\cdot)$ and $a(\cdot)$ are not assumed to be symmetric in spatial variables.
We study 
the  asymptotic behaviour of solutions to the Cauchy problem for equation \eqref{eq_eps} with the initial
condition 
\begin{equation}\label{ini_eps}
  u^\eps(x,0)=u_0(x),\quad u_0\in L^2(\mathbb R^d),
\end{equation}
as $\eps\to0$. 
It turns out that the asymptotic behaviour of solutions to the said Cauchy problem depends on the mixing properties of $\mu_\omega$.
In particular, if  $\mu_\omega$ has sufficiently good mixing properties,  then the solution of the original equation converges in law  in  moving coordinates  to a solution of the homogenized equation being a stochastic partial differential equation (SPDE).

It is worth noting that in the case of a symmetric kernel $a(x-y)\mu_\omega(x,y,t)$ the homogenized operator is a deterministic second order parabolic operator with constant coefficients, see Remark \ref{rem_nrlim}. Such a  significant change in the behavior of the limit equation occurs due to the lack of symmetry of the kernel. A convolution type equation with a non-symmetric kernel is a natural non-local counterpart of a convection-diffusion problem with a non-zero convection (drift) term.

Homogenization problems for  non-autonomous local convection-diffusion equations have been considered in \cite{KP},
the coefficients of the operators being random stationary in time and periodic in spatial variables.
 It was proved 
 that under the assumption that the coefficients of the equation have good mixing properties  the solution of the original problem converges in law to the solution of the limit stochastic partial differential equation in properly chosen moving coordinates.

The case of parabolic equations with a symmetric elliptic and a large potential (zero order term) has been treated in \cite{CKP}.

In this paper we follow the scheme developed in the proof of the homogenization results presented in \cite{KP},
however, due to the natural difference between differential and non-local convolution type operators the adaptation
of this scheme is rather non-trivial, it required new ideas and methods.

Averaging of parabolic differential equations of convection-diffusion type with large lower order terms and coefficients being  periodic both in spatial and temporal variables were studied in \cite{G, DP}.
In \cite{G}, the homogenization result was obtained under the assumption that the time oscillation is slower than the spacial one, and the effective drift is equal to zero.
In \cite{DP}, it was proved that the homogenization result holds in rapidly moving coordinates, and in the presence of zero order terms, homogenization takes place after factorization of solutions with the ground state of the corresponding cell problem.

\bigskip

In the autonomous non-local case when  $ \mu(\xi, \eta)$ in equation \eqref{eq_eps} is a periodic non-symmetric function of spatial variables and does not depend on time, the corresponding Cauchy problem admits homogenization in moving coordinates $x^\varepsilon = x-\frac{b}{\varepsilon} t$, where $b \in \X$ is an effective drift,  see \cite{AA}.
Since we borrow some constructions from this paper, we briefly recall here the model and the main result of \cite{AA}.
This paper deals with the Cauchy problems for an equation of the form \eqref{eq_eps} with the initial condition 
$u_0\in L^2(\mathbb R^d)$ under the assumption that 
$\mu(x,y)$ is periodic in $x$ and $y$ and does not depend on time. It is also assumed that $a(\cdot)$ is non-negative,
integrable and has  finite second moments, and $\mu$ satisfies the inequality $0<\alpha_-\leqslant \mu\leqslant\alpha_+$.
It was shown that, under these conditions, the said Cauchy problem admits homogenization in moving coordinates, i.e.
there is a constant vector $b\in\mathbb R^d$ such that for the solution $u^\eps$ of this problem the following representation is valid:
$$
u^\eps(x,t)=u^0\Big(x-\frac b\eps t,t\Big)+ R^\eps(x,t), \qquad\hbox{where }\ 
\lim\limits_{\eps\to0}\|R^\eps\|_{L^\infty(0,T;L^2(\mathbb R^d))}=0,
$$
and $u^0(x,t)$ is a solution of the Cauchy problem for the homogenized second order parabolic operator
 with constant coefficients that reads
$$
\partial_t u-\mathrm{div}\big(\Theta^{\mathrm{eff}}\nabla u\big)=0,\qquad u(x,0)=u_0(x),
$$
with the same initial condition $u_0$. 

The effective drift $b$ is defined by
\begin{equation}\label{v0}
 b \ = \  \int\limits_{\mathbb R^d}  \int\limits_{\mathbb T^d} a (\xi-q) \mu (\xi, q )  (\xi - q ) \ dq \, v_0 (\xi) d \xi,
\end{equation}
where $v_0 \in L^2(\mathbb{T}^d)$ is the solution of equation
\begin{equation}\label{FA1}
\int\limits_{\mathbb R^d} a (q - \xi) \mu (q,\xi ) \big( v_0 (q)
-v_0(\xi)\big) d q  = 0, \quad\int_{\mathbb T^d} v_0(\xi)d \xi =1.
\end{equation}
It is proved in \cite{AA} that there exist two positive constants $\gamma_1, \gamma_2>0$ such that
$0< \gamma_1 \le v_0 (\xi) \le \gamma_2 < \infty$ for all $\xi \in  \mathbb{T}^d$.

 The homogenization problem for non-autonomous convolution type equations with coefficients periodic
 in all variables including time  
has been addressed in \cite{PZh2023}.

\medskip
In the present work we consider the case of non-autonomous random in time equations.  More precisely, we assume that
the coefficient $\mu(\xi,z,s)$ in \eqref{eq_eps} is a stationary ergodic random function of  $s$ with values
in $L^\infty(\mathbb T^{2d})$. 
In the second section we provide the assumptions on the coefficients and formulate the main homogenization
results of this work.
Section three deals with formal asymptotic expansion of a solution and various auxiliary ''cell'' problems.
Solutions of these problems play crucial role in deriving the homogenized problem and proving the convergence.  
The proof of our homogenization results in the case of general stationary ergodic function $\mu$ is provided in
Section 4.
These results are then refined in Section 5 under the assumptions that $\mu(\cdot,s)$ has good mixing properties. 
The special case of $\mu_\omega(x,y,t)$ being the product $\mu_\omega(x,y,t)=\lambda_\omega(t)\hat\mu(x,y)$
is investigated in Section 6. 
The proof of technical statements is given in Appendices.

\section{Problem setup and main results}\label{sec_setup}

The goal of the present work is to study a homogenization problem for non-autonomous  parabolic type problems with non-symmetric kernels in stationary ergodic in time and periodic in space media:
\begin{equation}\label{ANA_eps}
\begin{array}{l}
H^\varepsilon_\omega u := \frac{\partial u}{\partial t} -  L^{\varepsilon}_\omega u = 0, \quad u(\cdot,0)=u_0\in
L^2(\mathbb R^d),\quad \mbox{where} \\[2mm]
(L^\varepsilon_\omega u)(x,t) \ = \ \frac{1}{\varepsilon^{d+2}} \int\limits_{\mathbb R^d} a\big(\frac{x-y}{\eps}\big) \mu_\omega\big(\frac{x}{\eps}, \frac{y}{\eps}; \frac{t}{\eps^2}\big) (u(y,t) - u(x,t)) dy,
\end{array}
\end{equation}
and $a(z)$ satisfies the following conditions:
\begin{equation}\label{M1}
  a(z)\geqslant 0, \quad \int_{\mathbb R^d}a(z)dz=1,\quad \int_{\mathbb R^d}|z|^2a(z)dz<\infty.
\end{equation}
We assume that the periodic media evolves in time in a random manner. { Namely, given a probability space  $(\Omega, {\cal F}, P)$, we define the random coefficient $\omega \mapsto \mu_\omega(x,y; t)$, $\omega\in\Omega$, 
as a measurable mapping from $\Omega$ to $L^\infty(\mathbb R^{2d}\times \mathbb R)$, and assume that 
$ \mu_\omega(x,y; t)$ is periodic in $x$ and $y$ with the period $[0,1)^d$ and that 
 there exist positive constants $\alpha_1$ and $\alpha_2$ such that
\begin{equation}\label{lm-random}
0<\alpha_1 \le  \mu_\omega (x,y; t) \le \alpha_2 \quad \mbox{ for all } \; x,y, \in \mathbb{R}^d, \; t \in \mathbb{R}.
\end{equation}
We then suppose that  the field  $ \mu_\omega(x, y; t)$ is {\bf stationary and ergodic} in time.
More precisely, it is assumed that there is a measure preserving dynamical system $ \tau_s:\Omega \to \Omega$,
$s\in\mathbb R$, such that:
$$
\mu_\omega(x, y; t+s)= \mu_{\tau_s \omega}(x,y;t), \qquad t,s \in \mathbb{R}, \; \omega \in \Omega.
$$
This property is referred to as stationarity. In particular, this implies that the law of the random function 
$\mu_{\omega}(x, y;t)$ is independent of $t$. Furthermore, we assume that the dynamical system $\tau_s$ is ergodic, 
that is if $A \subset \Omega$ is such that $\tau_s A=A$ for all $s \in \mathbb{R}$, then $P(A) = 0$ or $1$.

In addition to these general assumptions on $ \mu_\omega(x, y; t)$, in the second part of  Theorem \ref{MT} we also suppose that  $ \mu_\omega(x, y; t)$ satisfies good {\bf mixing conditions}. Under these mixing conditions a more precise
 result on the homogenization of problem \eqref{ANA_eps} is valid, see \eqref{th-4-II} below. We formulate here the mixing conditions in terms of the $\sigma$-algebra $\mathcal{F}_t$ associated with the random process $\mu_\omega$  following the definition from \cite{ShJ}. The mapping $(\omega,t) \to  \mu_\omega(x, y; t) \in 
 L^\infty(\mathbb{T}^{2d})$ is ${\cal F} \otimes \mathbb{R}$-measurable, and we set
$$
{\cal F}^\mu_{\le t} = \sigma \{ \mu(\cdot, \cdot, s): \ s \le t \}, \quad {\cal F}^\mu_{\ge t} = \sigma \{ \mu(\cdot, \cdot, s): \ s \ge t \}.
$$
The function $\alpha(t)$ defined as
\begin{equation}\label{alpha}
\alpha(t) = \alpha ({\cal F}^\mu_{\le 0}, \ {\cal F}^\mu_{\ge t} ) = \sup\limits_{A \in {\cal F}^\mu_{\le 0},\, B \in {\cal F}^\mu_{\ge t} } |P(A \cap B) - P(A) P(B)|
\end{equation}
is called the {strong mixing coefficient} of $\mu$. We assume that
\begin{equation}\label{alpha-1}
 \int\limits_0^{\infty}  \alpha^{1/2}(t)  dt < \infty.
\end{equation}

Denote by $u^{\varepsilon}_\omega(x,t)$ a solution of the evolution problem
\begin{equation}\label{th-2}
\frac{\partial u (x,t)}{\partial t} = L^{\varepsilon} u (x,t), \quad u(x,0) = u_0(x), \quad u_0 \in L^2(\mathbb R^d),
\end{equation}
$t \in [0, T]$, where the operator $ L^{\varepsilon}$ is defined in  \eqref{ANA_eps}.
The main result of the present work is the following statement:

\begin{theorem}\label{MT}
I. Let $a(z)$ satisfy condition \eqref{M1},
and assume that $\mu_\omega (x,y,t)$ is a stationary ergodic process with values in $L^\infty(\mathbb{T}^{2d})$ that
satisfies estimates \eqref{lm-random} for all $\omega \in \Omega$. 
Then there exists a positive definite symmetric matrix $\Theta$ and a vector-valued stationary ergodic
process $\beta_\omega(t) \in {\mathbb R^d}$ such that for any $T>0$ and for a.e. $\omega \in \Omega$
\begin{equation}\label{th-4-I}
\| u^{\varepsilon}_\omega ( x + \frac{b}{\eps}\ t + G^\eps_0 (t), \ t ) -
u^0 (x,\ t)\|_{L^{\infty}((0,T),\ L^2(\mathbb R^d) )}  \to 0, \quad \mbox{as } \; \varepsilon \to 0,
\end{equation}
where $u^0(x,t)$ is a solution of the Cauchy problem
\begin{equation}\label{th-3}
\frac{\partial u}{\partial t} = \Theta \cdot \nabla \nabla u, \quad u(x,0) = u_0(x), \quad u_0 \in L^2(\mathbb R^d), \quad t \in [0, T],
\end{equation}
$b=\mathbb{E}\beta_\omega(0)$, and
\begin{equation}\label{G0}
G^\varepsilon_0 (t)= \frac{1}{\eps} \int\limits_0^{t} \mathop{ \beta_\omega}\limits^{\circ}(\frac{s}{\varepsilon^2}) ds = \varepsilon \int\limits_0^{t/\varepsilon^2} \mathop{ \beta_\omega}\limits^{\circ} (\tau) d\tau = o(\eps^{-1}), \quad \mathop{ \beta_\omega}\limits^{\circ} (t) = \beta_\omega(t) - b.
\end{equation}

\medskip

II. If in addition the random field  $\mu_\omega (x,y; t), \ t \in \R,$ satisfies the mixing conditions \eqref{alpha-1}, then
\begin{equation}\label{th-4-II}
 u^{\varepsilon}_\omega ( x + \frac{b}{\eps}\ t, \ t )  \  \mathop{\rightarrow}\limits^{{\cal L}} \
u^0 (x -  \sigma W_t,\ t), \quad t \in [0, T], \qquad \mbox{as } \; \varepsilon \to 0,
\end{equation}
in the space $L^2([0,T];L^2(\mathbb R^d))$ equipped with the strong topology;
here $\mathop{\rightarrow}\limits^{{\cal L}} $ stands for the convergence in law, $W_t$ denotes a standard Wiener process in $\mathbb R^d$,  and
$$
\sigma \sigma^* = 2 \int\limits_0^{\infty} \mathbb{E} \big( \mathop{ \beta}\limits^{\circ}(0) \otimes \mathop{ \beta}\limits^{\circ}(t) \big) \, d t.
$$
\end{theorem}

\begin{remark}{\rm
It is worth noting that the function $v^0(x,t) =  u^0(x- \sigma W_t, \, t)$ is a solution of the following SPDE:
 \begin{equation}\label{SPDE-2}
\left\{
\begin{array}{l} \displaystyle
d v^0(x,t) = A^{\mathrm{eff}} \cdot \nabla\nabla v^0(x,t) dt -  \nabla v^0 \cdot \sigma \, d W_t \\[3mm]
\displaystyle v^0(x,0) = u^0(x,0) = u_0(x),
\end{array}
\right.
\end{equation}
with
$A^{\mathrm{eff}} = \Theta+ \frac12 \sigma \sigma^* $.
Indeed, it is sufficient to apply the It\^o formula 
to the function $u^0(x- \sigma W_t, \, t)$ and to use  \eqref{th-3}.
}
\end{remark}

\medskip

The more detailed formulation of Theorem \ref{MT} and its proof are given in Sections \ref{Proof-1} - \ref{Proof-3}. 
In particular, all the effective characteristics of the limit problem are expressed explicitly in terms of
solutions of auxiliary problems on the torus, see \eqref{b}, \eqref{Beta}, \eqref{Theta}.
%

\section{Proof of Theorem \ref{MT}: Ansatz and correctors}\label{Proof-1}

\subsection{Ansatz for solution}

In this section we construct an approximation for a solution  $u^\varepsilon$ of problem \eqref{th-2}. Denote by 
$\mathcal{S}(\mathbb{R}^d)$  the Schwartz class of functions in $\mathbb R^d$.
For a given $u \in C^{\infty}((0,T),{\cal{S}}(\mathbb R^d))$  we define the following ansatz:
\begin{equation}\label{w_eps}
\begin{array}{l}
\displaystyle
w^{\varepsilon}(x,t)  =  u ( x^\varepsilon\!,  t )
+ \varepsilon \varkappa_1 (\frac{x}{\varepsilon}, \frac{t}{\varepsilon^2})\cdot\! \nabla u  ( x^\varepsilon\!,  t )
+ \varepsilon^2 \varkappa_2 (\frac{x}{\varepsilon}, \frac{t}{\varepsilon^2})\cdot\!\nabla \nabla u  ( x^\varepsilon\!,t ),
\end{array}
\end{equation}
where
\begin{equation}\label{G}
x^\varepsilon = x -  \frac{1}{\varepsilon} G^\varepsilon(t), \qquad
G^\varepsilon(t)= \int\limits_0^{t} \beta(\frac{s}{\varepsilon^2}) ds = \varepsilon^2 \int\limits_0^{t/\varepsilon^2} \beta(\tau) d\tau,
\end{equation}
$\beta(s) = \beta_\omega(s) \in \mathbb{R}^d$ 
is a  bounded stationary ergodic vector process  with the mean
\begin{equation}\label{b}
\mathbb{E} \beta_\omega = b \in \mathbb{R}^d;
\end{equation}
$ \varkappa_1(\xi, s) = \varkappa_1^\omega(\xi, s)= \{ \varkappa^i_1(\xi,s), i = 1, \ldots, d \}$ and $\varkappa_2(\xi, s) = \varkappa_2^\omega(\xi, s) = \{ \varkappa^{ij}_2(\xi,s), i,j = 1, \ldots, d\}$ are stationary fields taking values in $(L^2(\mathbb T^d))^d$  and $ (L^2(\mathbb T^d))^{d^2}$ respectively. The process $\beta_\omega(s)$ as well as the random correctors  $\varkappa_1 (\xi, s), \,  \varkappa_2 (\xi, s)$  will be defined below.

Our goal is to construct the correctors $\varkappa_1$, $\varkappa_2$ and the moving coordinates $x^\eps$ in such a way
 that  $\partial_t w^{\varepsilon}(x,t)-(L^\eps w^{\varepsilon})(x,t)\approx \partial_t u(x,t)-\mathrm{div} \big(\Theta^{\mathrm{eff}} 
 \nabla u\big)(x,t) $.

In what follows for the sake of brevity we do not indicate explicitly the dependence on $\omega$, unless it leads to 
an ambiguity. 

Substituting $w_\eps$ for $u$ in \eqref{ANA_eps}  we have
\begin{equation}\label{Aw}
H^{\varepsilon}_\omega w^{\varepsilon}(x,t) \
= \ \frac{\partial w^{\varepsilon}(x, t)}{\partial t} - L^{\varepsilon}_\omega w^{\varepsilon}(x,t)
\end{equation}
with
\begin{equation}\label{Aw-1}
\begin{array}{l}
\displaystyle
\frac{\partial w^{\varepsilon}(x, t)}{\partial t} =
 - \frac{1}{\varepsilon} \beta(\frac{t}{\eps^2}) \cdot \nabla u (x^\varepsilon, t)  + \frac{\partial u}{\partial t} (x^\varepsilon, t)
\\[3mm] \displaystyle
+ \frac{1}{\varepsilon} \frac{\partial \varkappa_1}{\partial t} \big(\frac{x}{\varepsilon}, \frac{t}{\varepsilon^2} \big) \cdot \nabla u (x^\varepsilon, t)
+ \varepsilon \varkappa_1 \big(\frac{x}{\varepsilon}, \frac{t}{\varepsilon^2}\big) \otimes
\big( - \frac{1}{\varepsilon} \beta(\frac{t}{\eps^2}) \big) \cdot \nabla \nabla u (x^\varepsilon, t)
\\[3mm] \displaystyle
 +  \varepsilon \varkappa_1 \big(\frac{x}{\varepsilon}, \frac{t}{\varepsilon^2}\big)  \cdot \nabla \frac{\partial u}{\partial t} (x^\varepsilon, t)
 + \frac{\partial \varkappa_2}{\partial t} \big(\frac{x}{\varepsilon}, \frac{t}{\varepsilon^2} \big) \cdot \nabla \nabla u (x^\varepsilon, t)
\\[3mm] \displaystyle
+ \varepsilon^2 \varkappa_2 \big(\frac{x}{\varepsilon}, \frac{t}{\varepsilon^2}\big) \!\otimes\!
\big(\! - \!\frac{1}{\varepsilon} \beta(\frac{t}{\eps^2}) \big)\! \cdot \!\! \nabla \nabla \nabla u (x^\varepsilon\!, t)\! + \! \varepsilon^2 \varkappa_2 \big(\frac{x}{\varepsilon}, \frac{t}{\varepsilon^2}\big) \! \cdot \! \nabla \nabla \frac{\partial u}{\partial t} (x^\varepsilon\!, t)
 \end{array}
 \end{equation}
 and
\begin{equation}\label{Aw-2}
\begin{array}{l}
\displaystyle
(L^{\eps}_\omega w^{\varepsilon})(x,t) =  \frac{1}{\varepsilon^{d+2}} \int\limits_{\mathbb R^d} a\big(\frac{x-y}{\eps}\big) \mu_\omega\big(\frac{x}{\eps}, \frac{y}{\eps}; \frac{t}{\eps^2}\big) (w^{\varepsilon}(y,t) - w^{\varepsilon}(x,t)) dy
\\[3mm] \displaystyle = \
\frac{1}{\varepsilon^{d+2}} \int\limits_{\mathbb R^d} a \big( \frac{x-y}{\varepsilon} \big) \mu_\omega \big( \frac{x}{\varepsilon}, \frac{y}{\varepsilon}; \frac{t}{\eps^2} \big)
\bigg\{ u(y^\varepsilon,t)+ \varepsilon \varkappa_1 \big(\frac{y}{\varepsilon}, \frac{t}{\eps^2}\big)\cdot \nabla u(y^\varepsilon,t) +
\\[3mm] \displaystyle +\
\varepsilon^2 \varkappa_2 \big(\frac{y}{\varepsilon}, \frac{t}{\eps^2}\big)\cdot \nabla \nabla u(y^\varepsilon,t) -
u(x^\varepsilon,t)-\varepsilon \varkappa_1 \big(\frac{x}{\varepsilon},  \frac{t}{\eps^2} \big)\cdot \nabla u(x^\varepsilon,t)
\\[3mm] \displaystyle
 - \ \varepsilon^2 \varkappa_2 \big(\frac{x}{\varepsilon}, \frac{t}{\eps^2}\big)\cdot \nabla \nabla u(x^\varepsilon,t) \bigg\} dy.
\end{array}
\end{equation}
The symbols $\cdot$ and $\otimes$ stand for the scalar product and the tensor product,respectively.

We collect in \eqref{Aw-1} - \eqref{Aw-2} power-like terms and keep only the terms of order $\varepsilon^{-1}$ and $1$;
the higher order  terms form the  remainders $\phi^\varepsilon_{1,2}$. 
For $\frac{\partial w^{\varepsilon}}{\partial t}$ we obtain
\begin{equation}\label{Aw-t}
\begin{array}{l}
\displaystyle
\frac{\partial w^{\varepsilon}(x, t)}{\partial t} =
 - \frac{1}{\varepsilon}  \beta(\frac{t}{\eps^2})  \cdot \nabla u (x^\varepsilon, t)  +
 \frac{1}{\varepsilon} \frac{\partial \varkappa_1}{\partial t} \big(\frac{x}{\varepsilon}, \frac{t}{\varepsilon^2} \big) \cdot \nabla u (x^\varepsilon, t)
\\[3mm] \displaystyle
+ \frac{\partial u}{\partial t} (x^\varepsilon, t) - \varkappa_1 \big(\frac{x}{\varepsilon}, \frac{t}{\varepsilon^2}\big) \otimes
 \beta(\frac{t}{\eps^2})  \cdot \nabla \nabla u (x^\varepsilon, t)
\\[3mm] \displaystyle
 + \frac{\partial \varkappa_2}{\partial t} \big(\frac{x}{\varepsilon}, \frac{t}{\varepsilon^2} \big) \cdot \nabla \nabla u (x^\varepsilon, t)  + \phi^\varepsilon_1 (x,t),
 \end{array}
 \end{equation}
with
\begin{equation}\label{Aw-t-r}
\begin{array}{l}
\displaystyle
\phi^\varepsilon_1(x,t) =  \varepsilon \varkappa_1 \big(\frac{x}{\varepsilon}, \frac{t}{\varepsilon^2}\big)  \cdot \nabla \frac{\partial u}{\partial t} (x^\varepsilon, t) -
 \varepsilon \varkappa_2 \big(\frac{x}{\varepsilon}, \frac{t}{\varepsilon^2}\big) \otimes
 \beta(\frac{t}{\eps^2})  \cdot \nabla \nabla \nabla u (x^\varepsilon, t)
\\[3mm] \displaystyle
+ \varepsilon^2 \varkappa_2 \big(\frac{x}{\varepsilon}, \frac{t}{\varepsilon^2}\big)  \cdot \nabla \nabla \frac{\partial u}{\partial t} (x^\varepsilon, t)
\end{array}
\end{equation}

After change of variables $z = \frac{x-y}{\varepsilon} = \frac{x^\varepsilon - y^\varepsilon}{\varepsilon}$ we get
\begin{equation}\label{ml_1}
\begin{array}{l}\displaystyle
(L^{\varepsilon} w^{\varepsilon})(x,t) \ = \ \frac{1}{\varepsilon^{2}} \int\limits_{{\mathbb R}^d} dz \  a (z)  \mu \big(\frac{x}{\varepsilon}, \frac{x}{\varepsilon} -z;  \frac{t}{\eps^2} \big) \bigg\{ u(x^\varepsilon -\varepsilon z,t)
\\[3mm] \displaystyle
+ \varepsilon \varkappa_1 \big(\frac{x}{\varepsilon}-z;  \frac{t}{\eps^2} \big)\cdot \nabla u (x^\varepsilon -\varepsilon z,t)
+\,\varepsilon^2 \varkappa_2 \big( \frac{x}{\varepsilon}-z,  \frac{t}{\eps^2} \big)\cdot \nabla \nabla u(x^\varepsilon -\varepsilon z,t)
\\[3mm] \displaystyle
 - u(x^\varepsilon,t) -\varepsilon \varkappa_1 \big( \frac{x}{\varepsilon},  \frac{t}{\eps^2} \big)\cdot\nabla u(x^\varepsilon,t) - \varepsilon^2 \varkappa_2 \big(\frac{x}{\varepsilon},  \frac{t}{\eps^2} \big)\cdot \nabla \nabla u(x^\varepsilon,t) \bigg\}
\\[3mm] \displaystyle
= \frac{1}{\varepsilon} \nabla u(x^\varepsilon,t)\! \cdot\! \int\limits_{\mathbb R^d}  \Big\{ -z + \varkappa_1 \big(\frac{x}{\varepsilon}-z,  \frac{t}{\eps^2} \big) - \varkappa_1 \big(\frac{x}{\varepsilon},  \frac{t}{\eps^2}\big) \Big\}  a (z) \mu \big( \frac{x}{\varepsilon}, \frac{x}{\varepsilon} -z;  \frac{t}{\eps^2} \big) \, dz
\\[3mm]
\displaystyle
 +\, \nabla \nabla u (x^\varepsilon,t)\!\cdot\!\!  \int\limits_{\mathbb R^d}\! \Big\{ \frac12 z\!\otimes\!z\! - z \!\otimes\!\varkappa_1 \big(\frac{x}{\varepsilon}\!-\!z,  \frac{t}{\eps^2} \big)
\\[3mm]
\displaystyle
+ \, \varkappa_2 \big( \frac{x}{\varepsilon}\! -\! z,  \frac{t}{\eps^2} \big) \!- \varkappa_2 \big(\frac{x}{\varepsilon},  \frac{t}{\eps^2}\big)  \Big\} \  a (z) \mu \big(\frac{x}{\varepsilon}, \frac{x}{\varepsilon}\! -\!z;  \frac{t}{\eps^2} \big) \, dz
\, +\, \ \phi^\varepsilon_2 (x,t)
\end{array}
\end{equation}
with the remainder
\begin{equation}\label{14}
\begin{array}{l}
\displaystyle
\phi^\varepsilon_2 (x,t)  =   \int\limits_{\mathbb R^d}
\bigg\{ \int\limits_0^{1} \big( \nabla \nabla u(x^\varepsilon-\varepsilon zq, t) - \nabla \nabla u(x^\varepsilon,t) \big) \!\cdot\! z\!\otimes\!z \,(1-q) \ dq
\\[4mm]   \displaystyle
+\, \varepsilon \varkappa_1 \big(\frac{x}{\varepsilon}\!-\!z, \frac{t}{\eps^2} \big)\!\cdot\! \int\limits_0^{1}\!  \nabla \nabla \nabla u(x^\varepsilon\!-\!\varepsilon zq, t) z\!\otimes\!z (1\!-\!q) \, dq  \,
\\[4mm]   \displaystyle
- \, \varepsilon \varkappa_2 \big(\frac{x}{\varepsilon}\!-\!z, \frac{t}{\eps^2} \big) \!\cdot\! \int\limits_0^{1}\!  \nabla \nabla \nabla u(x^\varepsilon\!-\!\varepsilon zq, t)z  \, dq\!  \bigg\} \, a (z) \mu \big( \frac{x}{\varepsilon}, \frac{x}{\varepsilon}\! -\!z; \frac{t}{\eps^2} \big) \, dz.
\end{array}
\end{equation}
We denote 
\begin{equation}\label{fiplus}
\phi^\varepsilon \ = \ \phi^\varepsilon_1 \ - \ \phi^\varepsilon_2.
\end{equation}

\begin{proposition}\label{fi_0} Let $u \in C^{1}\big( (0,T), {\cal{S}}(\mathbb R^d) \big)$,
and assume that  $ \mu(\cdot)$ satisfies estimates (\ref{lm-random}). Assume, moreover, that all the components of $\varkappa_1(\xi, s)$ and $\varkappa_2(\xi, s)$ are elements of $L^\infty((0,+\infty);L^2(\mathbb T^d))$.
Then for the functions $ \phi^\varepsilon_1$ and $ \phi^\varepsilon_2$ given by \eqref{Aw-t-r} and \eqref{14} we have
\begin{equation}\label{fi}
\| \phi^\varepsilon_1 \|\big._{\infty} \ \to \ 0 \quad \mbox{ and } \quad \| \phi^\varepsilon_2 \|\big._{\infty} \ \to \ 0 \quad  \mbox{ as } \; \varepsilon \to 0,
\end{equation}
where $\| \cdot \|\big._{\infty}$ is the norm in $ L^{\infty}\big( (0,T), L^2 (\mathbb R^d) \big)$.
\end{proposition}

\begin{proof}
The desired convergence of $ \phi^\varepsilon_1$ immediately follows from representation \eqref{Aw-t-r} for this function.
The proof of the second relation in \eqref{fi}  is completely analogous to that of Proposition 5 in \cite{PZh}.
\end{proof}

\subsection{Terms of  order $\varepsilon^{-1}$}

Denote by $\xi=\frac{x}{\varepsilon}$ a "slow" variable on the period: $\xi \in \mathbb{T}^d$; then  $ \varkappa^i_1(\xi,s),  i = 1, \ldots, d,$ are functions on $(-\infty,+\infty)\times\mathbb T^d$.
Collecting all the terms of order $\varepsilon^{-1}$ in \eqref{Aw-t}, \eqref{ml_1} and equating them  to $0$, we get from \eqref{Aw} the following problem for the first corrector $\varkappa_1 (\xi, s), \xi  \in \mathbb{T}^d, s 
\in \mathbb{R}$: 
\begin{equation}\label{Fcorr}
 \frac{\partial \varkappa_1}{\partial s} (\xi, s )= \int\limits_{\mathbb R^d}  \Big( -z + \varkappa_1 (\xi-z, s) - \varkappa_1 (\xi, s) \Big)  a (z) \mu( \xi, \xi -z; s) \, dz + \beta(s),
 \end{equation}
written in the vector form; here the random field $\beta(s)=\beta_\omega(s)$ is unknown and should be determined as well. Since \eqref{Fcorr} is a system of uncoupled equations,
we can consider separately the equation for each component $\varkappa^i_1(\xi,s), \, i= 1, \ldots, d$, of the vector function $ \varkappa_1(\xi,s)$. It read 
\begin{equation}\label{Fcorr-1}
 \frac{\partial \varkappa^i_1}{\partial s} (\xi, s ) = A_\omega(s) \varkappa^i_1 (\xi, s) + f^i_\omega( \xi, s), \quad \varkappa_1 \in \big( L^2(\mathbb{T}^d)  \big)^d,
 \end{equation}
with
\begin{equation}\label{Acorr-1}
A_\omega(s) \varkappa^i_1 (\xi, s) = \int\limits_{\mathbb R^d}  \big( \varkappa^i_1 (\xi-z, s) - \varkappa^i_1 (\xi, s) \big)  a (z) \mu( \xi, \xi -z; s) \, dz,
 \end{equation}
\begin{equation}\label{fcorr-1}
f^i_\omega (\xi, s) =  - \int\limits_{\mathbb R^d}  z^i \, a (z) \, \mu( \xi, \xi -z; s) \, dz + \beta^i(s).
 \end{equation}


\subsection{A stationary solution of the auxiliary problem  \eqref{Fcorr-1}}

The main purpose of this section is to prove that there exists a uniquely defined stationary process $\beta(s)$ such that problem \eqref{Fcorr} has a stationary solution $\varkappa_1$. 
 As we noted above we can consider this problem for each component of the vector function $\varkappa_1$ separately.

We start with the study of the following Cauchy problems associated with the operators $H_\omega$ and $H_\omega^\ast$:
\begin{equation}\label{HC}
{\partial}_t u = A_\omega(t)u, \qquad u|_{t=0}= u_0(\xi) \in L^2(\mathbb{T}^d), \quad t\geqslant 0;
\end{equation}
\begin{equation}\label{HastC}
-\partial_t  p^N = A^\ast_\omega(t)p^N, \qquad p^N|_{t=N} = 1,  \quad t\leqslant N,
\end{equation}
where $N \in \mathbb{R}^+$, and we then let $N$ go to infinity. Observe that problem \eqref{HastC} is a parabolic problem with the reverse direction of time.

We formulate now several properties of solutions of the problems \eqref{HC} - \eqref{HastC} that will be exploited in the construction of the random corrector  $\varkappa_1 (\xi, t) $. All these properties hold realization-wise 
The symbols $u$ and $p^N$ stand for solutions of problems \eqref{HC} and \eqref{HastC}, respectively. 

\begin{lemma}\label{L1}
For each $t \in [0,N]$,
\begin{equation}\label{prop1}
\int\limits_{\mathbb{T}^d} u(\xi, t) p^N (\xi, t) d \xi = \int\limits_{\mathbb{T}^d} u_0(\xi )p^N(\xi, 0) d \xi =: C(N).
\end{equation}
\end{lemma}

\begin{proof}
For any solution $u$ to problem \eqref{HC} and for any solution $p^N$ of \eqref{HastC} we have
$$
\int\limits_{t_1}^{t_2} \int\limits_{\mathbb{T}^d} \frac{\partial}{\partial s} (u p^N)(\xi, s) ds d\xi =
 \int\limits_{t_1}^{t_2} \big[(A_\omega(s) u, p^N) - (u, A^\ast_\omega (s)  p^N)  \big] ds = 0.
$$
Consequently, $(u (\cdot, t), p^N (\cdot, t)) = C(N)$ for all $t \in [0,N]$.
\end{proof}

\begin{remark}
Let us note that formula \eqref{prop1} remains valid if one replaces
the terminal  condition in \eqref{HastC} with an arbitrary terminal condition $p(N,\cdot)=\zeta(\cdot)\in L^2(\mathbb T^d)$.
 \end{remark}

Since $u(\xi,t)= 1$ is a solution to problem \eqref{HC} with $u_0=1$, we immediately obtain the following property of $p^N$:
\begin{lemma}\label{L2}
For each $t \in [0,N]$,
\begin{equation}\label{prop2}
\int\limits_{\mathbb{T}^d} p^N (\xi, t) d \xi = \int\limits_{\mathbb{T}^d} p^N(\xi, N) d \xi = 1.
\end{equation}
\end{lemma}

\begin{lemma}\label{L2-bis}
\begin{equation}\label{prop2-bis}
p^N (\xi, t) >0 \qquad \forall \;\; N, \;  t < N, \; \xi \in \mathbb{T}^d.
\end{equation}
\end{lemma}

The proof of this Lemma is provided in Appendix 1.

The following three statements are crucial for constructing the correctors and obtaining a priori estimates.
\begin{lemma}\label{L3}
For each $N \in \mathbb{R}$ and for all $t \in (- \infty, N]$ there exist positive deterministic constants $\pi_1$ and $\pi_2$ such that
\begin{equation}\label{prop3}
0< \pi_1 \le p^N (\xi, t) \le \pi_2.
\end{equation}
\end{lemma}


\begin{lemma}\label{L4}
There exist positive deterministic constants $\gamma_0$ and $C_4$ such that
\begin{equation}\label{prop4}
\| u (\cdot, s) - C(N)\|_{L^2(\mathbb{T}^d)} \le C_4 \, e^{-\gamma_0 s} \, \| u_0\|_{L^2(\mathbb{T}^d)},
\end{equation}
for all $s \in (0, N]$, where the constant $C(N)$ is defined by \eqref{prop1}.
In particular, if $s=N$, then
\begin{equation}\label{prop4-1}
\| u (\cdot, N) - C(N)\|_{L^2(\mathbb{T}^d)} \le C_4 \, e^{-\gamma_0 N} \, \| u_0\|_{L^2(\mathbb{T}^d)}.
\end{equation}
\end{lemma}


\begin{lemma}\label{L4-add}
There exists a limit
\begin{equation}\label{prop4-add}
\lim\limits_{N \to \infty} C(N) = C_{\infty}, \quad C_{\infty} = \lim\limits_{N \to \infty} \int\limits_{\mathbb{T}^d} u(\xi, N) d \xi,
\end{equation}
and
\begin{equation}\label{prop4-add1}
\| u (\cdot, N) - C_{\infty}\|_{L^2(\mathbb{T}^d)} \le C_5 \, e^{-\gamma_0 N} \, \| u_0\|_{L^2(\mathbb{T}^d)},
\end{equation}
where $C_5$ is a deterministic constant independent of $N$.
\end{lemma}

The proof of Lemmata \ref{L3}--\ref{L4-add} is given in Appendix 2.

\medskip
We conclude that a solution of problem \eqref{HC} converges at the exponential rate to a constant, and the
convergence is uniform for $u_0$ from the unit ball in $L^2(\mathbb T^d)$.  \\
Observe that the constants $C(N)$ and $C_\infty$ are random.

\begin{lemma}\label{L4-bis}
For any $K>0$ and $s<N$,
\begin{equation}\label{prop4-bis}
\| p^N (\cdot, s) - p^{N+K}(\cdot, s)\|_{L^2(\mathbb{T}^d)} \le C_6 \, e^{-\gamma_0 (N-s)},
\end{equation}
where $C_6$ is a deterministic constant independent of $N$.
\end{lemma}

The proof of Lemma \ref{L4-bis} is given in Appendix 2.

\medskip

Next we show that the equation $-\partial_t  p = A^\ast_\omega(s)p$ has a stationary solution, this solution is unique up to a multiplicative constant.

\begin{proposition}\label{uniq}
There exists a limit
\begin{equation}\label{uniq-1}
p_\infty(\xi, s) = \lim\limits_{N \to \infty} p^N(\xi, s), \quad \xi \in \mathbb{T}^d, \; s \in \mathbb{R}.
\end{equation}
The function $p_\infty(\xi,s)$ is a stationary solution of the equation $-\partial_t  p = A^\ast_\omega(s)p$.
Moreover,
\begin{equation}\label{uniq-2}
\int\limits_{\mathbb{T}^d} p_\infty (\xi, s) d \xi =1 \;\; \forall s \in \mathbb{R}; \qquad
0< \pi_1 \le p_\infty (\xi, s) \le \pi_2,
\end{equation}
$p_\infty (\xi,s) \in C(\mathbb{R}, \, L^2(\mathbb{T}^d))$. Under the first condition in \eqref{uniq-2} 
the stationary solution is unique.
\end{proposition}

The proof of this statement is also given in Appendix 2.

\medskip

Now we turn to problem \eqref{Fcorr-1}--\eqref{Acorr-1} and 
formulate conditions that ensure the existence of its stationary solution. 
Given a process $g_\omega(\xi,s)$ with values in $L^2(\mathbb T^d)$
consider the equation
\begin{equation}\label{Du-1}
 \frac{\partial v}{\partial s} (\xi, s ) = A(s) v (\xi, s) + g_\omega( \xi, s),
 \quad (\xi, s) \in \mathbb{T}^d \times \mathbb{R},
 \end{equation}
 with 
\begin{equation}\label{Du-2}
A(s) v (\xi, s) = \int\limits_{\mathbb R^d}  \big( v (\xi-z, s) - v (\xi, s) \big)  a (z) \mu( \xi, \xi -z; s) \, dz.
\end{equation} 
\begin{lemma}\label{Du}
Let  $g_\omega( \xi, s)$ be a stationary process such that a.s.  $g_\omega( \xi, s)
\in L^\infty((-\infty,+\infty); L^2(\mathbb T^d))$, and assume that the process $\big(g_\omega( \xi, s), 
p_\infty(\xi,s)\big)$ is also stationary, where $p_\infty(\xi, s)$ is the stationary solution of the equation $-\partial_t  p = A^\ast(s)p$ defined in \eqref{uniq-1}. 
If a.s.
\begin{equation}\label{Du-3}
\int\limits_{\mathbb T^d}  g_\omega( \xi, s) \, p_\infty(\xi, s) \, d\xi = 0  \quad \forall \ s \in \mathbb{R},
 \end{equation}
then equation \eqref{Du-1} has 
a stationary solution $v^\omega (\xi,s) \in  L^\infty((-\infty,+\infty);L^2(\mathbb{T}^d))$, which is unique up to an additive constant.
\end{lemma}

For the proof of Lemma \ref{Du} see Appendix 2.

\medskip
The statement of Lemma \ref{Du} suggests the following choice of $\beta_\omega(s)$ in \eqref{Fcorr}:  
\begin{equation}\label{Beta}
 \beta_\omega(s) =  \int\limits_{\mathbb R^d} \int\limits_{\mathbb T^d} z \, a (z) \, \mu( \xi, \xi -z; s) \,  p_\infty(\xi, s) \, d\xi dz.
 \end{equation}
Indeed, considering \eqref{Fcorr-1}, \eqref{fcorr-1} and \eqref{uniq-2}, by Lemma \ref{Du} we conclude that 
for $\beta$ defined in \eqref{Beta} equation \eqref{Fcorr} has a stationary solution $\varkappa_1 \in L^\infty\big( -\infty,+\infty;(L^2(\mathbb{T}^d))^d  \big) $, which is uniquely defined up to an additive constant.

  We then define the first corrector $\varkappa_1$ as a stationary solution of  \eqref{Fcorr} and observe that
all terms of order $\varepsilon^{-1}$  in \eqref{Aw-t}, \eqref{ml_1} vanish.

\begin{remark}\label{rem_nrlim}
Observe that in the symmetric case $a(z)=a(-z)$ and $\mu(\xi,z,t)=\mu(z,\xi,t)$  the function $p_\infty(\xi,t)=1$
and $\beta(t)=0$. Therefore, in this case the homogenized problem is deterministic, it reads
$$
\partial_t u=\mathrm{div}\big(\Theta\nabla u\big)\ \hbox{\rm in }\mathbb R^d\times(0,T],\qquad u(x,0)=u_0.
$$     
 \end{remark}

\section{Proof of Theorem \ref{MT}. Part I.}
\label{Proof-2}

\subsection{Terms of order $\varepsilon^0$}

Using \eqref{Aw-t} and \eqref{ml_1} we collect all the terms of order $\varepsilon^{0}$  in \eqref{Aw}:
\begin{equation}\label{II-1}
\frac{\partial u}{\partial t} (x^\varepsilon, t)
- {\cal B}_\omega  \big(\frac{x}{\varepsilon}, \frac{t}{\varepsilon^2}\big)\! \cdot \! \nabla \nabla u (x^\varepsilon, t),
\end{equation}
where 
\begin{equation}\label{II-3}
{\cal B}_\omega  (\xi, s) =
 -  \frac{\partial \varkappa_2}{\partial s} (\xi,s) + A_\omega(s) \varkappa_2 (\xi,s) + \tilde g_\omega (\xi,s), \quad \xi = \frac{x}{\varepsilon}, \ s= \frac{t}{\varepsilon^2}.
\end{equation}
Here the operator $ A(s)$  is the same as in \eqref{Acorr-1} or \eqref{Du-2}, and
\begin{equation}\label{II-4}
\begin{array}{l}
\displaystyle
 \tilde g_\omega (\xi,s) =
 \varkappa_1 (\xi,s)\! \otimes \! \beta_\omega(s)
 \\[3mm]
\displaystyle
  +  \int\limits_{\mathbb R^d}\! \big( \frac12 z\!\otimes\!z\! - z \!\otimes\!\varkappa_1 (\xi\!-\!z, s )\big)
\  a (z) \mu (\xi, \xi\! -\!z; s ) \, dz.
 \end{array}
 \end{equation}
 Letting 
 \begin{equation}\label{II-6}
\Theta^{ij}_\omega(s) = \int\limits_{\mathbb T^d}  \tilde g^{ij}_\omega( \xi, s) \, p_\infty(\xi, s, \omega) \, d\xi,
 \end{equation}
and applying Lemma \ref{Du} we conclude that the equation
\begin{equation}\label{II-5}
 \frac{\partial v}{\partial s} (\xi, s ) = A_\omega(s) v (\xi, s) + \tilde g^{ij}_\omega( \xi, s) - \Theta^{ij}_\omega(s),
  \end{equation}
has a stationary solution which is unique up to an additive constant. Denoting this solution by $\varkappa^{ij}_2$ we obtain the matrix function $\varkappa_2=\{\varkappa^{ij}_2\}_{i,j=1}^d\in 
L^\infty(-\infty,+\infty;L^2(\mathbb{T}^d)^{d^2})$. Substituting  $\varkappa_2$ in \eqref{II-3} yields
\begin{equation}\label{II-7}
{\cal B}_\omega  (\xi, s) =  \Theta_\omega(s), \quad \Theta_\omega(s) = \{ \Theta^{ij}_\omega(s)\}_{i,j=1}^d,
\end{equation}
where $\Theta^{ij}_\omega(s)$ is the stationary field defined by \eqref{II-6}.

Finally,  the terms of order $\varepsilon^0$ in \eqref{Aw} take the form
\begin{equation}\label{II-8}
\frac{\partial u}{\partial t} (x^\varepsilon, t)
- \Theta_\omega  \big( \frac{t}{\varepsilon^2}\big) \cdot  \nabla \nabla u (x^\varepsilon, t)
\end{equation}
with the stationary field  $\Theta_\omega(\cdot)$ defined in \eqref{II-6}. 
\begin{lemma}\label{l_posidef}
  The matrix $\Theta_\omega$ is positive definite, there exists a constant $\lambda>0$ such that a.s.
   $\lambda|\zeta|^2 \leqslant \Theta_\omega(t)\zeta\cdot\zeta\leqslant \lambda^{-1}|\zeta|^2$
   for all $t\in\mathbb R$  and $\zeta\in\mathbb R^d$.
\end{lemma}

This statement can be proved in the same way as Proposition 5.1 in \cite{AA}. Consequently, the matrix
\begin{equation}\label{Theta}
\Theta = \mathbb{E} \Theta_\omega (s)
\end{equation}
is also positive definite, and does not depend on $s$ due to the stationarity of $\Theta_\omega(\cdot)$.

\subsection{Partial homogenization}

Let $\rho^\eps$ be a solution of the following problem
\begin{equation}\label{apriori-1}
\partial_t  \rho^\eps(x,t) = \Theta_\omega \big(\frac{t}{\varepsilon^2} \big)  \cdot  \nabla \nabla \rho^\eps (x, t), \quad \rho^\eps(x,0) = u_0 \in {\cal S}(\mathbb{R}^d).
\end{equation}
Differentiating this equation in spatial variables and considering the fact that $\Theta_\omega$ does not depend on $x$
one can easily check that  $\rho^\eps(x,t) \in C^1((0,T), {\cal S}(\mathbb{R}^d))$ for any $T$.
  We then substitute $\rho^\eps$ for $u$ in \eqref{w_eps} and define an ansatz $w^\varepsilon$ by formula \eqref{w_eps}
 with $\beta_\omega(s)$ given by \eqref{Beta} and  $\varkappa_1(\xi,s)$ and  $\varkappa_2(\xi,s)$ being stationary
 solutions of equations \eqref{Fcorr-1}--\eqref{fcorr-1} and \eqref{II-4}--\eqref{II-5}, respectively.
  It follows from \eqref{II-8} that $w^\varepsilon$ satisfies the following equation
\begin{equation}\label{w_eps_bis}
\begin{array}{l}
\displaystyle
\partial_t w^{\varepsilon} -  L_\omega^{\varepsilon} w^{\varepsilon} =
\partial_t \rho^\eps(x^\varepsilon,t) -
 \Theta_\omega  \big(\frac{t}{\varepsilon^2} \big) \cdot \nabla \nabla \rho^\eps(x^\varepsilon,t) \ + \ \phi^\varepsilon(x^\varepsilon,t) =
\phi^\varepsilon(x^\varepsilon,t),
\\[3mm] \displaystyle
 w^{\varepsilon}(x,0) = u_0(x) + \psi^\varepsilon(x),
 \end{array}
\end{equation}
where $x^\varepsilon$ is introduced in \eqref{G} with $\beta_\omega(s)$ given by \eqref{Beta}, the remainder term $\phi^\varepsilon(x^\varepsilon,t)$ is defined in \eqref{fiplus}, and
$$
\psi^\varepsilon (x) \ = \ \varepsilon \varkappa_1 (\frac{x}{\varepsilon},0)\cdot \nabla u_0(x) + \varepsilon^2 \varkappa_2 (\frac{x}{\varepsilon},0)\cdot \nabla \nabla u_0 (x). 
$$
It is strightforward to check that for any $u_0\in \mathcal{S}(\mathbb R^d)$
\begin{equation}\label{small_ini}
  \|\psi^\eps\|_{L^2(\mathbb R^d)}\ \to\ 0,\quad\hbox{as }\eps\to0.
\end{equation}
Consequently, the difference  $v^\varepsilon = w^\varepsilon - u^\varepsilon$, where $u^\varepsilon$ is the solution of \eqref{th-2}, 
satisfies the following problem:
\begin{equation}\label{v_eps_bis}
\partial_t v^{\varepsilon}(x,t) -  L_\omega^{\varepsilon} v^{\varepsilon}(x,t) = \phi^\varepsilon(x^\varepsilon,t), \quad v^{\varepsilon}(x,0) =  \psi^\varepsilon(x).
\end{equation}
Estimates for $v^{\varepsilon}(\cdot)$ rely on the following statement.
\begin{proposition}\label{prop_apriori}
A solution of the problem
\begin{equation}\label{au_apriori}
\partial_t \Xi^{\varepsilon}(x,t) -  L_\omega^{\varepsilon} \Xi^{\varepsilon}(x,t) = f(x,t), \quad \Xi^{\varepsilon}(x,0) =  g_0(x),
\end{equation}
with $f\in L^2(0,T;L^2(\mathbb R^d))$ and $g_0\in L^2(\mathbb R^d)$ admits the following upper bound
$$
\|\Xi\|_{L^\infty(0,T;L^2(\mathbb R^d))}\leqslant C_1(T)\|f\|_{L^2(0,T;L^2(\mathbb R^d))}
+C_2\|g_0\|_{L^2(\mathbb R^d)}
$$
with constants $C_1(T)$ and $C_2$ that do not depend on $\eps$.
\end{proposition}

\begin{proof}
Multiplying the equation in \eqref{au_apriori} by $p_\infty\big(\frac x\eps,\frac t{\eps^2}\big)\Xi(x,t)$ and integrating the resulting relation over $\mathbb R^d\times (0,t)$, $t\leqslant T$, we obtain
$$
\begin{array}{c}
\displaystyle
\frac12\int\limits_0^t\int\limits_{\mathbb R^d}p_\infty\big(\frac x\eps,\frac s{\eps^2}\big)\partial_s((\Xi(x,s))^2) dxds
-\int\limits_0^t\int\limits_{\mathbb R^d}p_\infty\big(\frac x\eps,\frac s{\eps^2}\big)\Xi(x,s)L_\omega^{\varepsilon} \Xi^{\varepsilon}(x,s) dxds\\[3mm]
\displaystyle
=\int\limits_0^t\int\limits_{\mathbb R^d}p_\infty\big(\frac x\eps,\frac s{\eps^2}\big)f(x,s)\Xi(x,s)dxds.
\end{array}
$$
Taking into account the relation $-\partial_s p_\infty(y,s)=A^*(s)p_\infty(y,s)$ and making the same 
transformations as in the proof of  \cite[Proposition 6.1]{AA} we derive the relation
$$
\begin{array}{c}
\displaystyle
\int\limits_{\mathbb R^d}p_\infty\big(\frac x\eps,\frac t{\eps^2}\big)(\Xi(x,t))^2 dx\\[3mm]
\displaystyle 
+\int\limits_0^t\int\limits_{\mathbb R^d}a\Big(\frac{x-y}\eps\Big)
\mu\Big(\frac x\eps,\frac y\eps,\frac s{\eps^2}\Big)p_\infty\big(\frac x\eps,\frac s{\eps^2}\big)
(\Xi(x,s)- \Xi(y,s))^2 dxds\\[3mm]
\displaystyle
=\int\limits_{\mathbb R^d}p_\infty\big(\frac x\eps,0\big)(g_0(x))^2 dx
+2\int\limits_0^t\int\limits_{\mathbb R^d}p_\infty\big(\frac x\eps,\frac s{\eps^2}\big)f(x,s)\Xi(x,s)dxds.
\end{array}
$$
In view of \eqref{uniq-2} the desirted upper bound follows from the latter relation by the Gronwall theorem.
\end{proof}
Since both $\|\phi^\varepsilon \|_{\infty}$ and $\| \psi^\varepsilon \|_{L^2(\mathbb{R}^d)}$  tend to zero as $\eps\to0$ (see Proposition \ref{fi_0} and estimate \eqref{small_ini}), then applying Proposition \ref{prop_apriori} to problem 
\eqref{w_eps_bis} we conclude 
that
\begin{equation}\label{apriori-2}
\|v^\eps\|_{\infty} = \| w^\varepsilon - u^\varepsilon\|_{\infty}  \to 0 \quad \mbox{ as } \;  \eps \to 0.
\end{equation}
On the other hand, considering the structure of the ansatz $w^\eps$ in \eqref{w_eps} one can easily show that
\begin{equation}\label{apriori-3}
 \| w^\varepsilon(x,t) - \rho^\varepsilon(x^\eps, t)\|_{\infty}  \to 0 \quad \mbox{ as } \;  \eps \to 0.
\end{equation}
Combining the last two limit relations yields
\begin{equation}\label{apriori-3biss}
 \| u^\varepsilon(x,t) - \rho^\varepsilon(x^\eps, t)\|_{\infty}  \to 0 \quad \mbox{ as } \;  \eps \to 0.
\end{equation}

Denote $\Theta = \mathbb{E} \Theta_\omega(\cdot)$ and let $u^0(x,t)$ be a solution to the Cauchy problem
\begin{equation}\label{equa_u0}
\frac{\partial u}{\partial t} = \Theta \cdot \nabla \nabla u, 
\quad u(x,0) = u_0(x), 
\quad t \in [0, T].
\end{equation}
Our next goal is to show that the solution $\rho^\eps$ of problem \eqref{apriori-1} converges to $u^0$
as $\eps\to0$.
\begin{lemma}\label{l_closedness}
  With probability one for any $u_0\in \mathcal{S}(\mathbb R^d)$ the family $\{\rho^\eps(x,t)\}$ converges to $u^0(x,t)$ in $C(0,T;L^2(\mathbb R^d))$ as $\eps\to0$. 
\end{lemma}
\begin{proof}
Differentiating equation \eqref{apriori-1} in spatial variables 
it is straightforward to obtain the following estimates for $\rho^\eps$:
\begin{equation}\label{ee_compaa}
\|\rho^\eps\|_{L^\infty(0,T;H^1(\mathbb R^d))}+\|\partial_t \rho^\eps\|_{L^\infty(0,T;H^{1}(\mathbb R^d))}
\leqslant C\|u_0\|_{H^3(\mathbb R^d)},
\end{equation}
here the constant $C$ does not depend on $\eps$. Denote $B_R=\{x\in\mathbb R^d\,:\,|x|\leqslant R\}$. Since for any $R>0$ the space $W^{1,\infty}(0,T;H^1(B_R))$ is compactly imbedded to $C(0,T;L^2(B_R)$, estimate \eqref{ee_compaa} implies that
the family $\{\rho^\eps\}$ is compact in $C(0,T;L^2(B_R))$. Combining this compactness result with
the Aronson estimates, see \cite{Arons}, we conclude that $\{\rho^\eps\}$ is compact in $C(0,T;L^2(\mathbb R^d))$.
For a converging subsequence of $\{\rho^\eps\}$ denote its limit by $\rho^0$. Let $\psi(x,t)$ be a 
$C^\infty(0,T;C_0^\infty(\mathbb R^d))$ function such that $\psi(x,T)=0$. The integral identity of problem \eqref{apriori-1} reads
$$
-\int_0^T\int_{\mathbb R^d}\rho^\eps \partial_t\psi\,dxdt=\int_{\mathbb R^d} u_0 \psi(\cdot,0)\,dx+
\int_0^T\int_{\mathbb R^d}\rho^\eps \Theta\Big(\frac t{\eps^2}\Big)\cdot\nabla\nabla\psi\,dxdt
$$
Passing to the limit as $\eps\to0$
in this integral identity 
we obtain the integral identity
of problem  \eqref{equa_u0}. Therefore, $\rho^0=u^0$ and the whole family $\{\rho^\eps\}$ converges to $u^0$:
 \begin{equation}\label{apriori-4}
 \| u^0 - \rho^\varepsilon\|_{\infty}  \to 0 \quad \mbox{ as } \;  \eps \to 0.
\end{equation}
\end{proof}

Since $\|\rho^\eps(x,t)-u^0(x,t)\|_\infty=\|\rho^\eps(x^\eps,t)-u^0(x^\eps,t)\|_\infty$, then \eqref{apriori-2} - \eqref{apriori-4} yield the convergence
\begin{equation}\label{apriori-5}
 \| u^\eps (x,t) - u^0(x^\varepsilon,t)\|_{\infty}  \to 0 \quad \mbox{ as } \;  \eps \to 0,
\end{equation}
where $x^\eps$ is defined by formula \eqref{G}.
Recalling that $\beta(s)$ introduced in \eqref{Beta} is a stationary ergodic process, by the Birkhoff ergodic theorem
we have
\begin{equation}\label{apriori-6}
x^\eps =  x -  \varepsilon \int\limits_0^{t/\varepsilon^2} \beta(\tau) d\tau  =  x - \frac{b}{\eps} t -  G^\eps_0 (t),
\end{equation}
where
\begin{equation}\label{apriori-7}
b = \mathbb{E} \beta = \lim\limits_{T \to \infty} \frac{1}{T} \int\limits_0^{T} \beta(\tau) d\tau,
\end{equation}
and $ G^\eps_0 (t)$ defined by \eqref{G0} is such that a.s. for any $t>0$ we have $\eps G^\eps_0 (t)\to 0$ as $\eps\to 0$.
Thus, we have proved convergence \eqref{th-4-I} for any
$u_0 \in {\cal S}(\mathbb{R}^d)$.

Approximating any $L^2$ initial function by a sequence of ${\cal S}(\mathbb R^d)$ functions and taking into account the a priori estimates obtained in Proposition \ref{prop_apriori}
and similar estimates for the limit problem in \eqref{equa_u0} we derive the first statement of Theorem \ref{MT}.

\section{Proof of Theorem \ref{MT}. Part II} 
\label{Proof-3}

In this section we assume that mixing condition \eqref{alpha-1} holds. In this case  the asymptotic behaviour 
of the process $x^\varepsilon $ given by \eqref{apriori-6} - \eqref{apriori-7} as $\eps \to 0$ can be described  more
precisely. 
We recall that
$\beta (s) = \beta_\omega (s)$ is defined in \eqref{Beta} as follows
\begin{equation*}\label{Beta*}
 \beta(s) =  \int\limits_{\mathbb R^d} \int\limits_{\mathbb T^d} z \, a (z) \, \mu_\omega( \xi, \xi -z; s) \,  p_\infty(\xi, s) \, d\xi dz,
 \end{equation*}
and it is a bounded stationary ergodic process with mean $b=\mathbb{E} \beta$.

\begin{lemma}\label{L-3}
For the process $G^\varepsilon_0 (t)$ the invariance principle (functional central limit theorem) holds
i.e.
\begin{equation}\label{G0-lemma}
G^\varepsilon_0 (t)= 
 \varepsilon \int\limits_0^{t/\varepsilon^2}  \mathop{ \beta_\omega}\limits^{\circ} (\tau) d\tau \  \mathop{\rightarrow}\limits^{{\mathcal L}} \ \sigma \, W_t, \qquad \mathop{ \beta_\omega}\limits^{\circ} (t) = \beta_\omega(t) - b,
\end{equation}
where $W_t$ is the $d$-dimensional Wiener process, and the matrix $\sigma$ is defined by the formula:
\begin{equation}\label{sigma}
\sigma \sigma^* = 2 \int\limits_0^{\infty} \mathbb{E} \big(  \mathop{ \beta_\omega}\limits^{\circ} (0) \, \otimes \mathop{ \beta_\omega}\limits^{\circ} (t) \big) \, dt.
\end{equation}
\end{lemma}

\begin{proof}
We first show that the following key estimate holds: 
\begin{equation}\label{G0Proof-1}
\|\, \mathbb{E}( \mathop{ \beta}\limits^{\circ} (0) | \, \mathcal{F}^{\beta}_{\ge t}) \, \|_2  \le C_1 e^{-C_2 t} + C_3 \alpha^{1/2}(\frac{t}{2}),
\end{equation}
where $\| \cdot \|_2 = \| \cdot \|_{L^2(\Omega)}$, $ \ \mathcal{F}^{\beta}_{\geqslant t}$ is the $\sigma$ -algebra generated by $\beta(s)$ with $s\geqslant t$ and $\ C_1, \, C_2, \, C_3$ are constants. Notice that  
$\mathcal{F}^{\mathop{ \beta}\limits^{\circ}}_{\geqslant t}=\mathcal{F}^{\beta}_{\geqslant t}$.

To obtain \eqref{G0Proof-1} we rewrite $ \mathop{ \beta}\limits^{\circ} (0)$ as the sum
$$
 \mathop{ \beta}\limits^{\circ} (0) =  \mathop{ \beta_1}\limits^{\circ} (0) +  \mathop{ \beta_2}\limits^{\circ} (0),
$$
where
\begin{equation}\label{G0Proof-2}
 \mathop{ \beta_1}\limits^{\circ} (0) =  \int\limits_{\mathbb R^d} \int\limits_{\mathbb T^d} a (z)\big( z \, \mu( \xi, \xi -z; 0) - b \big) \,  p_1(\xi, 0) \, d\xi dz,
 \end{equation}
\begin{equation}\label{G0Proof-3}
 \mathop{ \beta_2}\limits^{\circ} (0) =  \int\limits_{\mathbb R^d} \int\limits_{\mathbb T^d}  a (z)\big( z \, \mu( \xi, \xi -z; 0) - b \big) \,(p_\infty(\xi,0) -  p_1(\xi, 0)) \, d\xi dz,
 \end{equation}
 and $ p_1(\xi, s), \, s \le \frac{t}{2},$ is a solution of the problem
$$
-\partial_s  p_1 = A^\ast(s)p_1, \qquad p_1|_{s=\frac{t}{2}} = 1.
$$
By Lemma \ref{L4-bis} and Proposition \ref{uniq}, applying the Schwartz inequality we have $\|p_\infty(\cdot,s) - p_1(\cdot,s)\|_{L^2(\mathbb T^d)} \le 
\tilde C_1 e^{-\gamma(\frac t2-s)}$, therefore,
\begin{equation}\label{G0Proof-4}
\| \, \mathbb{E}( \mathop{ \beta_2}\limits^{\circ} (0) | \, \mathcal{F}^{\beta}_{\ge t}) \,\|_2 \le \| \, \mathop{ \beta_2}\limits^{\circ} (0) \,\|_2 \le C_1 e^{-\frac \gamma2 t}.
\end{equation}

Notice that by the definition of $\beta(s)$ and by the construction of $p_1$ we have
$$
{\mathcal{F}}^\beta_{\ge t} \ = \ {\mathcal{F}}^\mu_{\ge t} \quad \mbox{ and } \quad \mathop{ \beta_1}\limits^{\circ} (0)
\ \hbox{is } {\mathcal{F}}^\mu_{\le {t/2}}\hbox{ measurable}.
$$
 Consequently, due to \eqref{alpha}, by \cite[Chapter 8, Lemma 3.102]{ShJ} the following inequality holds:
 \begin{equation}\label{G0Proof-5}
\| \, \mathbb{E}( \mathop{ \beta_1}\limits^{\circ} (0) | \, \mathcal{F}^{\beta}_{\ge t}) \, \|_2  = \| \, \mathbb{E}( \mathop{ \beta_1}\limits^{\circ} (0) | \, \mathcal{F}^{\mu}_{\ge t}) \, \|_2  \le C_3 \alpha^{1/2}
\Big(\frac{t}{2}\Big).
\end{equation}
Thus, \eqref{G0Proof-4} - \eqref{G0Proof-5} imply estimate \eqref{G0Proof-1}. Consequently, if \eqref{alpha-1}  holds, then
 \begin{equation}\label{G0Proof-6}
\int\limits_0^{\infty} \| \, \mathbb{E}( \mathop{ \beta}\limits^{\circ} (0) | \, \mathcal{F}^{\beta}_{\ge t}) \, \|_2 \, dt < \infty.
\end{equation}

According to \cite[Theorem 3.79 (Chapter VIII),]{ShJ}, inequality  \eqref{G0Proof-6} ensures the invariance principle for the process $\eta(t) = \mathop{ \beta}\limits^{\circ} (-t)$. Indeed, since
$$
\mathbb{E}( \mathop{ \beta}\limits^{\circ} (0) | \, \mathcal{F}^{\beta}_{\ge t}) = \mathbb{E}( \eta (0) | \, \mathcal{F}^{\eta}_{\le -t}) = \mathbb{E}( \eta(t) | \, \mathcal{F}^{\eta}_{\le 0}),
$$
then \eqref{G0Proof-6} implies
$$
\int\limits_0^{\infty} \| \, \mathbb{E}( \eta(t) | \, \mathcal{F}^{\eta}_{\le 0}) \, \|_2 \, dt < \infty.
$$
Then the invariance principle holds for the process $\eta(t)$, and consequently, for the process  $\eta(-t) = \mathop{ \beta}\limits^{\circ} (t)$. 
This completes the proof of Lemma.
\end{proof}

By Lemma \ref{L-3} we have 
$$
G^\eps_0 (t) \mathop{\rightarrow}\limits^{{\mathcal L}} \ \sigma \, W_t \quad \mbox{as} \quad \eps \to 0.
$$
From this relation, taking into account \eqref{apriori-6} and the smoothness of the function $u^0(x,t)$, we  
derive that
$$
\textstyle
u^0\big(x-G^\eps_0(t),t) \mathop{\rightarrow}\limits^{{\mathcal L}} \ u^0\big(x-\sigma W_t,t)
$$
in the space $L^2(0,T; L^2(\mathbb R^d))$ equipped with the strong topology. Observe that \eqref{apriori-5} 
can be written in the form  $\|u^\varepsilon \big( x+\frac{b}{\eps}t, \, t \big)-u^0(x-G^\eps_0(t),t)\|_\infty\to0$.
Therefore, representing $u^\eps\big(x+\frac b\eps t,t\big)=u^0(x-G^\eps_0(t),t)+
\big[u^\eps\big(x+\frac b\eps t,t\big)-u^0(x-G^\eps_0(t),t)\big]$, we obtain
 \begin{equation}\label{SPDE-1}
u^\varepsilon \big( x+\frac{b}{\eps}t, \, t \big) \  \mathop{\rightarrow}\limits^{{\mathcal L}} \ u^0(x- \sigma W_t, \, t), \quad\hbox{as } \eps \to 0
\end{equation}
The second part of Theorem \ref{MT} is proved.

\section{The case:  $ \mu_\omega(x,y; t) = \lambda_\omega(t) \, \mu^0(x,y)$. }
\label{sec_product}

In this section we consider the special case of problem \eqref{ANA_eps} when $\mu_\omega(x,y; t)$ has the product structure:  $ \mu_\omega(x,y; t) = \lambda_\omega(t) \, \mu^0(x,y)$ with a 
stationary positive process $ \lambda_\omega(t) $ and deterministic periodic function $\mu^0$. The main result of this section reads:

\begin{theorem}\label{MT-special}
Let the functions $a(z)$ and $\mu^0(x,y)$ satisfy conditions \eqref{M1} - \eqref{lm-random}, and assume that
 $\lambda_\omega(\tau), \ \tau \in \R,$ is a stationary ergodic positive process such that
\begin{equation}\label{th-1}
\begin{array}{l}
1) \ \ m:= \mathbb{E} \lambda_\omega(0) <\infty, \\[2mm]
2) \ \ \mathbb{E} (\lambda_\omega(0) - m)^2 < \infty, \\[2mm]
3) \ \int\limits_0^{\infty} \| \mathbb{E}((\lambda_\omega(\tau)-m) | \mathcal{F}^\lambda_{\leqslant 0} ) \|_2 \, d\tau < \infty.
\end{array}
\end{equation}
Then 
there exists a constant vector $b\in\mathbb R^d$ and a positive definite symmetric constant matrix $\Theta$ such that
for any $u_0\in L^2(\mathbb R^d)$ and any $T>0$ we have 
\begin{equation}\label{th-4}
 u^{\varepsilon}_\omega ( x + \frac{b}{\eps}\ t, \ t )  \  \mathop{\rightarrow}\limits^{{\cal L}} \
u^0 (x - b \sigma W_t,\ t), \quad \mbox{as } \; \varepsilon \to 0, \qquad t \in [0, T];
\end{equation}
here $u^\eps(x,t)$ is the solution of problem \eqref{ANA_eps}, 
$u^0(x,t)$ is the solution to the homogenized problem
$$
\partial_t u=\mathrm{div}\big(\Theta\nabla u\big)=0 \ \ \hbox{in }\mathbb R^d\times(0,T],\quad u(x,0)=u_0(x), 
$$
the symbol $\mathop{\rightarrow}\limits^{{\cal L}} $ stands the the convergence in law in the space $L^\infty(0,T;L^2(\mathbb R^d))$ equipped with the strong topology,  $ W_t$ is a standard Wiener process,  and
$$
\sigma^2 = 2 \int\limits_0^{\infty} \mathbb{E} \big( (\lambda(0)-m) \, (\lambda(\tau)-m) \big) \, d \tau.
$$
\end{theorem}

\begin{proof}
In the case under consideration  problem \eqref{ANA_eps} takes the form 
\begin{equation}\label{pbm}
\textstyle
\partial_t u(x,t)= \lambda\big(\frac t{\eps^2}\big)\, L^\eps u(x,t), \qquad u(x,0)=u_0(x),
\end{equation}
with
$$
L^\eps u(x,t)=\frac 1{\eps^{d+2}}\int_{\mathbb R^d}a\Big(\frac{x-y}\eps\Big)\mu^0\Big(\frac{x}\eps,\frac{y}\eps\Big)\big(
u(y,t)-u(x,t)\big)dy.
$$ 
We first assume that $u_0$ is a $C_0^\infty(\mathbb R^d)$ function
and  want to make a suitable change of time in the equation in \eqref{pbm} so that \eqref{pbm} is reduced  to an autonomous evolution problem.
The structure of \eqref{pbm} suggests the following time change:
\begin{equation}\label{chatime1}
s(t) = \int_0^t \lambda\Big(\frac s{\eps^2}\Big) d s=\eps^2\int_0^{\frac t{\eps^2}} \lambda(s) d s, 
 \qquad \tilde u(x,s) = u(x, t(s)).
\end{equation}
Then
$$
\frac{\partial \tilde u(x,s)}{\partial s} = \frac{\partial u(x,t)}{\partial t}\ \frac{dt}{ds} =
 \lambda\Big(\frac t{\eps^2}\Big)\, L^\eps \tilde u(x,s) \frac{dt}{ds} = L^\eps \tilde u(x,s).
$$
Thus $\tilde u$ is the solution of the  equation
\begin{equation}\label{pbm-s}
\frac{\partial \tilde u(x,s)}{\partial s}  = L^\eps \tilde u (x,s),\quad \tilde u(x,0)=u_0(x),
\end{equation}
and we come to the homogenization problem for  autonomous parabolic equation. 

\medskip


As was mentioned in the Introduction, according to \cite{AA} there exists a constant vector $ b$ and a positive definite symmetric constant matrix $\Theta$ such that 
\begin{equation}\label{conv-tilde}
\| \tilde u^\varepsilon (x, s) \ - \  u^0(x - \frac{ b}{\varepsilon}s,s)\|\big._{L^\infty(0,T;L^2(\mathbb T^d))}
 \ \to \ 0, \qquad \varepsilon \to 0
\end{equation}
where 
$ u^0(x,s)$ is the solution of the limit problem 
$$
\partial_t u=\mathrm{div}\big(\Theta\nabla u\big)=0, \ \quad u(x,0)=u_0.
$$
Without loss of generality in what follows we assume that $m = 1$. Letting
$\eta_\omega (\tau) = \lambda_\omega(\tau) -1, \ \tau \in \R$, we have 
\begin{equation}\label{chatime}
s(t) = 
\int_0^t \lambda_\omega(\frac{s}{\varepsilon^2}) d s  =  t + \delta_\omega^\varepsilon (t)
\ \ \hbox{with }\ 
\delta_\omega^\varepsilon (t) = \int_0^t  \eta_\omega\Big(\frac{s}{\varepsilon^2}\Big) ds. 
\end{equation}


\begin{lemma}\label{l_6_1}
Under the assumptions in \eqref{th-1}
 for any $t>0$ almost surely
\begin{equation}\label{conv1}
\delta_\omega^\varepsilon (t) = \varepsilon^2 \int_0^{t/\varepsilon^2}  \eta_\omega(\tau) d \tau \to 0, \qquad \mbox{as } \; \varepsilon \to 0,
\end{equation}
and for any $T>0$ the convergence $\delta_\omega^\varepsilon (t) \to 0 $ is   uniform for $t \in [0,T]$.
Furthermore,
\begin{equation}\label{conv2}
\frac{1}{\varepsilon} \delta_\omega^\varepsilon (t) =  \varepsilon \int_{0}^{t/\varepsilon^2}  \eta_\omega(\tau) d \tau \ \mathop{\rightarrow}\limits^{{\cal L}} \ \sigma W_t, \qquad \mbox{as } \; \varepsilon \to 0,
\end{equation}
where $\mathop{\rightarrow}\limits^{{\cal L}} $ means the convergence in law in the space of continuous functions 
with values in $\mathbb R^d$. $ W_t$ is the standard Wiener process, and   
$$
\sigma^2 = 2 \int\limits_0^{\infty} \mathbb{E} (\eta(0) \, \eta(\tau)) \, d \tau.
$$

\end{lemma}

\begin{proof}
The first statement is an immediate consequence of  the law of large numbers. 
The second statement follows from \cite{ShJ}, Theorem 3.79, Chapter VIII.
\end{proof}

%

From \eqref{conv-tilde} we derive that almost surely
\begin{equation}\label{conv-back}
\|  u^\varepsilon (x, t) \ - \  u^0(x - \frac{ b}{\varepsilon}s(t),s(t))\|\big._{L^\infty(0,T;L^2(\mathbb T^d))}
 \ \to \ 0, \quad\hbox{as } \varepsilon \to 0,
\end{equation}
or, equivalently,
\begin{equation}\label{conv-back}
\|  u^\varepsilon (x+\frac{ b}{\varepsilon}t, t) \ - \  u^0(x - \frac{ b}{\varepsilon}\delta^\varepsilon_\omega(t),t+\delta^\varepsilon_\omega(t))\|\big._{L^\infty(0,T;L^2(\mathbb T^d))}
 \ \to \ 0, \ \ \hbox{as } \varepsilon \to 0.
\end{equation}
By Lemma \ref{l_6_1} taking into account the continuity of the function $u^0(x,t)$, we obtain
\begin{equation}\label{Law1}
 u^0 (x - \frac{b}{\varepsilon} \ \delta^\varepsilon_\omega(t),\ t+ \delta^\varepsilon_\omega(t))  \ \mathop{\rightarrow}\limits^{{\cal L}} \  u^0 (x - b \sigma W_t, t)
\end{equation}
in the space $L^\infty(0,T;L^2(\mathbb R^d))$.
Combining the last two limit relations we obtain the following convergence in law:
\begin{equation}\label{Law2}
u^\varepsilon_\omega (x + \frac{b}{\varepsilon} \ t, \ t)  \ \mathop{\rightarrow}\limits^{{\cal L}} \
 u^0 (x - b \sigma W_t,\ t) \qquad \mbox{as } \; \varepsilon \to 0
\end{equation}
in $L^\infty(0,T;L^2(\mathbb R^d))$.
This yeilds \eqref{th-4} for a smooth initial condition $u_0$. For the generic $u_0\in L^2(\mathbb R^d)$ the convergence 
can be  justified by approximation.
\end{proof}

\section{Appendix 1. Proof of Lemmas \ref{L2-bis} and \ref{L3}}

It is clear that the function $\tilde p^N(\xi,s):=p^N(\xi,N-s)$, where $p^N(\xi,s)$ is a solution to problem \eqref{HastC}, satisfies the equation
\begin{equation}\label{app1-1}
\partial_s  p(\xi,s) = A^\ast_\omega(N-s)p(\xi,s), \qquad p|_{s=0} = 1, \quad s \ge 0,
\end{equation}
with
\begin{equation}\label{app1-2}
\begin{array}{l}
\displaystyle
A^\ast_\omega(s)p (\xi,s) = \int\limits_{\mathbb{T}^d} \hat a(\eta - \xi) \mu_\omega (\eta,\xi;s)p(\eta,s) d \eta - G(\xi,s)p(\xi,s),
\\[2mm] \displaystyle
 G(\xi,s )= \int\limits_{\mathbb{T}^d} \hat a( \xi- \eta) \mu_\omega (\xi, \eta;s) d \eta.
 \end{array}
\end{equation}
\medskip

\noindent
{\it Proof of Lemma \ref{L2-bis}.} \\
Using two-sided bound \eqref{lm-random} on $\mu_\omega (\eta,\xi;s)$ we conclude that
the solution $\tilde p^N (\xi,s)$ of problem \eqref{app1-1}  can be bounded from below by the solution $R(\xi, s)$ of the following problem
\begin{equation}\label{app1-2bis}
 \partial_s R (\xi,s) = \alpha_1 \int\limits_{\mathbb{T}^d} \hat a(\eta - \xi) R(\eta,s) d \eta - \alpha_2 R(\xi, s), \quad  R(\xi,0) = 1,
\end{equation}
where $\alpha_1, \, \alpha_2$ are the same constants as in \eqref{lm-random}. Denote the operator on the right hand side of  \eqref{app1-2bis} by $A_\alpha$:
$$
A_\alpha R (\xi,s)= \alpha_1 \int\limits_{\mathbb{T}^d} \hat a(\eta - \xi) R(\eta,s) d \eta - \alpha_2 R(\xi, s).
$$
Then $R (\xi,s) = e^{s \, A_\alpha}1, \ s>0,$ and the Trotter formula for the sum of two bounded operators implies that $R(\xi, s)\ge 0$. The strict inequality relies on  \cite[Lemma 4.2]{AA} which states that there exist $k \in \mathbb{N}$ and $\gamma_0>0$ such that
\begin{equation*}
\hat a^{\ast k} (\xi) \ge \gamma_0, \quad \forall \xi \in \mathbb{T}^d.
\end{equation*}
Thus we obtain
$$
\tilde p^N (\xi,s) \ge R (\xi,s) >0.
$$

\bigskip

\noindent
{\it  Proof of Lemma \ref{L3}.} \\
I. {\it Lower bound}. For an arbitrary $t\geqslant 1$   on the interval $s \in [t-1,t]$ the solution $\tilde p^N(\xi,s)$ of problem \eqref{app1-1} coincides with the solution of the following problem
\begin{equation}\label{app1-3}
\partial_s  p = A^\ast_\omega(N-s)p, \quad s \in [t-1,t];  \qquad p|\big._{s=t-1} = \tilde p^N(\xi, t-1).
\end{equation}
After the change of unknown function
$$
q(\xi, s) = p(\xi, s) e^{\ \int\limits_{t-1}^s G(\xi,N- r) dr}, \quad 
 s \in [t-1,t],
$$
making straightforward rearrangements we conclude that for $s\geqslant t-1$ the function $q(\xi, s)$ is a solution to the problem
\begin{equation}\label{app1-4}
\begin{array}{l}
\displaystyle
 \partial_s q (\xi,s) = \int\limits_{\mathbb{T}^d} \hat a(\eta - \xi) \mu_\omega (\eta,\xi;N-s)   e^{\ \int\limits_{t-1}^s (G(\xi,N- r) - G(\eta,N-r)) dr} q(\eta,s) d \eta,
\\[2mm] \displaystyle
q(\xi, t-1) = \tilde p^N(\xi, t-1).
 \end{array}
\end{equation}
From \eqref{lm-random} and \eqref{app1-2} it follows that 
$$
e^{\ \int\limits_{t-1}^s (G(\xi, N-r) - G(\eta,N-r)) dr} \geqslant e^{-(\alpha_2-\alpha_1)}, \quad  s \in [t-1,t].
$$
Consequently for $ s \in [t-1,t]$ the solution $q (\xi,s)$ of problem \eqref{app1-4}  can be estimated from below by the solution $Q(\xi, s)$ of the following problem:
\begin{equation}\label{app1-5}
 \partial_s Q (\xi,s) = \alpha_1 e^{- (\alpha_2-\alpha_1)}\!\! \int\limits_{\mathbb{T}^d} \hat a(\eta - \xi) Q(\eta,s) d \eta,
\quad
Q(\xi, t-1) = \tilde p^N(\xi, t-1),
\end{equation}
where $\alpha_1$ and $\alpha_2$ are the same as in \eqref{lm-random}. The solution $Q (\xi,s), \ s \in [t-1,t], $ admits the representation
$$
 Q (\xi,s) = \tilde p^N(\xi, t-1) + \sum_{k=1}^{\infty} \frac{(s-(t-1))^k \ \alpha_1^k \ e^{-k (\alpha_2-\alpha_1)}}{k!} \ \hat a^{\ast k} \ast \tilde p^N(\xi, t-1).
$$
By Lemma 4.2 in \cite{AA} there exist $k_0 \in \mathbb{N}$ and $\gamma_0>0$ such that
\begin{equation}\label{app1-6}
\hat a^{\ast k_0} (\xi) \ge \gamma_0, \quad \forall \xi \in \mathbb{T}^d.
\end{equation}
Recalling \eqref{prop2} we derive from the latter inequality the estimate
$$
 Q (\xi,t) \ge \frac{\alpha_1^{k_0} \ e^{-k_0 (\alpha_2-\alpha_1)}}{k_0!} \gamma_0>0. 
$$
Therefore, for all $\xi \in \mathbb{T}^d$ we have
$$
\tilde p^N(\xi, t) = q(\xi,t) \ e^{ - \int\limits_{t-1}^t G(\xi, r) dr} \ge 
\frac{\alpha_1^{k_0} \ e^{-k_0 (\alpha_2-\alpha_1)}}{k_0!} \gamma_0 e^{-\alpha_2}=: \pi_1 >0.
$$
This yields the lower bound in \eqref{prop3} for $t\geqslant 1$.  For $s\in [0,1]$ we have
$$
 Q (\xi,s) = 1 + \sum_{k=1}^{\infty} \frac{s^k \ \alpha_1^k \ e^{-k (\alpha_2-\alpha_1)}}{k!} \ \hat a^{\ast k} \ast 1
 \geqslant 1,
$$
and the desired lower bound follows.

\medskip
\noindent
II. {\it Upper bound}.
From the definition of the operator $A(s)$ it follows that for any $s\in\mathbb R$ this operator is bounded
in $L^\infty(\mathbb T^d)$ and 
$$
\|A(s)\|\big._{L^\infty(\mathbb T^d)\to L^\infty(\mathbb T^d)}\leqslant 2\alpha_2.
$$
Then according to  \cite{Krein}, Chapter II, \S2 we have $\tilde p^N\in C(0,+\infty;L^\infty(\mathbb T^d))$. Since
$$
\|\tilde p^N(\cdot,t)\|_{L^\infty(\mathbb T^d)}=
\Big\|1+\int_0^t (A(N-s)\tilde p^N)(\cdot,s)ds-\int_0^t G(\cdot,s)\tilde p^N)(\cdot,s)ds\Big\|_{L^\infty(\mathbb T^d)}
$$
$$
\leqslant 1+2\alpha_2\int_0^t \|\tilde p^N(\cdot,s)\|_{L^\infty(\mathbb T^d)} ds,
$$
then by the Gronwall theorem $\|\tilde p^N(\cdot,t)\|_{L^\infty(\mathbb T^d)}\leqslant e^{2\alpha_2 t}$.

Consider an arbitrary interval $[0,T]$ and denote 
$M_T=\max\limits_{0\leqslant t\leqslant T} \|\tilde p^N(\cdot,t)\|_{L^\infty(\mathbb T^d)}$. 
Then, for any $t,\,s\in[0,T]$,
\begin{equation}\label{infy_cont}
 \|\tilde p^N(\cdot,t)-\tilde p^N(\cdot,s)\|_{L^\infty(\mathbb T^d)}\leqslant 2\alpha_2 M_T |t-s|.
\end{equation}
In particular, if $t_0\in[0,T]$ is a point where the maximum is attained and $S_1\subset\mathbb T^d$ is a set
of positive measure such that $\tilde p^N(\xi,t_0)\geqslant \big(1-\frac1{10}\alpha_1\big) M_T$ for all $\xi\in S_1$,
then 
\begin{equation}\label{infy_lobou}
\tilde p^N(\xi,t)\geqslant \frac12 M_T \quad\hbox{for all }(\xi,t)\in S_1\times\big[t_0-\frac 1{5\alpha_2},
 t_0\big].
\end{equation}

Our next goal is to estimate the integral
$$
J(\xi):=\int_{t_1}^{t_0}\int_{\mathbb T^d}\hat a(\xi-z)\mu(\xi,z,s)\tilde p^N(z,s)dzds,
$$
here $t_1$ stands for $t_0-\frac 1{5\alpha_2}$.  Denote $S_2(t)=\{\xi\in\mathbb T^d\,:\, \tilde p^N(\xi,t)\geqslant
\sqrt{M_T}\}$, and $S_3(t)=\mathbb T^d\setminus S_2(t)$.  In view of \eqref{prop2} by the Chebysheff inequality for
each $t\in[t_1,t_0]$ we obtain
$$
|S_2(t)|\leqslant \frac1{\sqrt{M_T}}.
$$
Since $\hat a(\cdot)$ is integrable, 
$$
\nu(M_T):=\sup\limits_{t\in[t_1,t_0]}\int_{S_2(t)}\hat a(\xi-z)dz=\sup\limits_{t\in[t_1,t_0]}\int_{\xi-S_2(t)}
\hat a(z)dz\to0,
$$
as $M_T\to\infty$. Then the integral $J(\xi)$ admits the following estimate:
$$
J(\xi)\leqslant \alpha_2\int_{t_1}^{t_0}\int_{\mathbb T^d}\hat a(\xi-z)\tilde p^N(z,s)dzds
$$
$$
=\alpha_2\int_{t_1}^{t_0}\int_{S_2(s)}\hat a(\xi-z)\tilde p^N(z,s)dzds+
\alpha_2\int_{t_1}^{t_0}\int_{S_3(s)}\hat a(\xi-z)\tilde p^N(z,s)dzds
$$
$$
\leqslant \alpha_2(t_0-t_1)\big(\nu(M_T)M_T+\sqrt{M_T}\big)=\frac{1}{5}\big(\nu(M_T)M_T+\sqrt{M_T}\big).
$$
Combining this estimate with \eqref{infy_lobou} for any $\xi\in S_1$ we obtain
$$ 
M_T-\frac{\alpha_1}{10}M_T\leqslant \tilde p^N(\xi,t_0)= \tilde p^N(\xi,t_1)+J(\xi)-\int_{t_1}^{t_0}G(\xi,s)\tilde p^N(\xi,s)ds
$$
$$
\leqslant M_T+\frac15\big(\nu(M_T)M_T+\sqrt{M_T}\big)-\frac12\alpha_1 M_T.
$$
After elementary rearrangement we have
$$
2\alpha_1 M_T\leqslant \big(\nu(M_T)M_T+\sqrt{M_T}\big).
$$
For large $M_T$ this inequality is contradictory. This yields the desired upper bound.

\section{Appendix 2. Proof of Lemmas \ref{L4} - \ref{Du}}

\noindent
{\it Proof of Lemma \ref{L4}.} \\
Multiplying equation  \eqref{HC} by $u(\xi, s)p^N(\xi,s)$ and integrating the resulting relation over $\mathbb T^d$ we have
\begin{equation}\label{L4-1}
\int\limits_{\mathbb{T}^d} (\partial_s u) \, u \, p^N \, d \xi = \int\limits_{\mathbb{T}^d} (A_\omega(s) u) \, u \, p^N \, d \xi.
\end{equation}
Using the relation $\partial_s(u^2 p^N)= 2(\partial_s u) u\, p^N + u^2 \, \partial_s p^N$, one can rewrite  \eqref{L4-1}
as follows:
\begin{equation}\label{L4-2}
\frac12 \int\limits_{\mathbb{T}^d} \partial_s (u^2 p^N) d \xi = \int\limits_{\mathbb{T}^d} (A_\omega(s) u) \, u \, p^N \, d \xi + \frac12 \int\limits_{\mathbb{T}^d}  u^2  \partial_s  p^N d \xi.
\end{equation}
Considering equation \eqref{HastC} we then rearrange the integral on the right-hand side of \eqref{L4-2}:
$$
\int\limits_{\mathbb{T}^d} \int\limits_{\mathbb{T}^d} \hat a(\xi - \eta) \mu_\omega(\xi,\eta; s) (u(\eta,s) - u(\xi,s)) \, u(\xi,s) \, p^N(\xi, s) d\eta d \xi + \frac12 \int\limits_{\mathbb{T}^d}  u^2(\xi,s)  \partial_s  p^N(\xi, s) d \xi
$$
$$
=\int\limits_{\mathbb{T}^d} \int\limits_{\mathbb{T}^d}\hat a(\xi - \eta) \mu_\omega(\xi,\eta; s) (u(\eta,s) - u(\xi,s)) \, u(\xi,s) \, p^N(\xi, s) \, d \eta \, d \xi
$$
$$
- \frac12 \int\limits_{\mathbb{T}^d} \int\limits_{\mathbb{T}^d} u^2(\xi,s) \hat a( \eta - \xi) \mu_\omega(\eta, \xi; s)  p^N(\eta, s) d \eta \, d \xi
$$
$$
+ \frac12 \int\limits_{\mathbb{T}^d} \int\limits_{\mathbb{T}^d} u^2(\xi,s) \hat a(\xi - \eta) \mu_\omega(\xi,\eta; s)  \, p^N(\xi, s) \, d \eta \, d \xi
$$
$$
= \int\limits_{\mathbb{T}^d} \int\limits_{\mathbb{T}^d}\hat a(\xi - \eta) \mu_\omega(\xi,\eta; s) u(\eta,s) u(\xi,s) p^N(\xi, s) \, d \eta \, d \xi
$$
$$
- \frac12 \int\limits_{\mathbb{T}^d} \int\limits_{\mathbb{T}^d}\hat a(\eta - \xi) \mu_\omega(\eta, \xi; s)  u^2(\xi,s)  p^N(\eta, s) d \eta \, d \xi
$$
$$
- \frac12 \int\limits_{\mathbb{T}^d} \int\limits_{\mathbb{T}^d} \hat a(\xi - \eta) \mu_\omega(\xi,\eta; s) u^2(\xi,s)  p^N(\xi, s) \, d \eta \, d \xi
$$
\begin{equation}\label{L4-3}
= - \frac12 \int\limits_{\mathbb{T}^d} \int\limits_{\mathbb{T}^d} \hat a(\xi - \eta) \mu_\omega(\xi,\eta; s) (u(\eta,s) - u(\xi,s))^2 \, p^N(\xi, s) \, d \eta \, d \xi.
\end{equation}
Denote by ${\cal A}^{sym}$ the operator in $L^2(\mathbb{T}^d)$ with a symmetric kernel $\hat a^{sym}(z) = \frac12 \big( \hat a(z) + \hat a (-z) \big)$:
$$
{\cal A}^{sym} u(\xi) =  \int\limits_{\mathbb{T}^d} \hat a^{sym}(\xi- \eta)(u( \eta) - u(\xi)) d \eta.
$$
The operator ${\cal A}^{sym}$ is the difference of a compact operator and the identity operator $I$. This operator has an eigenvalue at 0 with the corresponding eigenfunction equal to $1$, and the rest of the spectrum is negative.
By the Krein-Rutman theorem $0$ is a simple eigenvalue. The distance from $0$ to the rest of the spectrum is strictly positive. These spectral properties of ${\cal A}^{sym}$ imply that for any $u\in L^2(\mathbb{T}^d)$ we have:
\begin{equation}\label{L4-4}
 \big( -{\cal A}^{sym}u, u \big) \ge C_1 \| \tilde u \|_{L^2(\mathbb{T}^d)},\quad C_1>0,
\end{equation}
where $\tilde u  = u - (u, 1)$. 
Using \eqref{L4-4} we derive the following estimate for the integral on the right-hand side of  \eqref{L4-3}:
$$
\frac12 \int\limits_{\mathbb{T}^d} \int\limits_{\mathbb{T}^d} \hat a(\xi - \eta) \mu_\omega(\xi,\eta; s) (u(\eta,s) - u(\xi,s))^2 \, p^N(\xi, s) \, d \eta \, d \xi
$$
$$
\geqslant \frac 12 \alpha_1 \, \pi_1 \, \big( -{\cal A}^{sym}u, u \big) \geqslant C_2  \| \tilde u \|_{L^2(\mathbb{T}^d)},\quad 
C_2>0.
$$
We also introduce a function $\hat u = u - C(N)$. By Lemmas \ref{L1} - \ref{L2},  $C(N)= (u, p^N)$
and $(1, p^N)=1$, therefore $\hat u \perp p^N$.
 Observe that $\hat u $ might have  non-zero projections both on the space of constants and on the space orthogonal
 to constants. Moreover, the projection of $\hat u$ on the space orthogonal to constants coincides with $\tilde u$:
\begin{equation}\label{hat_ineq}
\hat u - (\hat u, 1) = u - C(N) - (u-C(N), 1) = u - (u,1) = \tilde u.
\end{equation}
\begin{lemma}\label{l_impineq}
There exists a constant $C_3<1$ such that
\begin{equation}\label{L4-5}
\| \tilde u \|_{L^2(\mathbb{T}^d)} \ge C_3 \| \hat u \|_{L^2(\mathbb{T}^d)}
\end{equation}
\end{lemma}
\begin{proof}
First we show that $(\hat u,1)<C_0\|\hat u\|_{L^2(\mathbb T^d)}$ for some constant $C_0<1$ that does not depend on $u$. Without loss of generality we assume that $\|\hat u\|_{L^2(\mathbb T^d)}=1$.  Then
$$
\|(1-\hat u)\|^2_{L^2(\mathbb T^d)}=1+\|\hat u\|^2_{L^2(\mathbb T^d)}-2(1,\hat u)=
2\big[\|\hat u\|^2_{L^2(\mathbb T^d)}-(1,\hat u)\big] 
$$
Using the properties of $p^N$ we have
$$
1=(1,p^N)= (p^N,\hat u)+(p^N,1-\hat u)=(p^N,1-\hat u) \leqslant \|p^N\|_{L^2(\mathbb T^d)}
\|1-\hat u\|_{L^2(\mathbb T^d)}
$$
Combining the last two relations yields
$$
1\leqslant \|p^N\|^2_{L^2(\mathbb T^d)}\|1-\hat u\|^2_{L^2(\mathbb T^d)}\leqslant 
2\pi_2^2\big[\|\hat u\|^2_{L^2(\mathbb T^d)}-(1,\hat u)\big]
$$
Therefore, 
$$
\|\hat u\|_{L^2(\mathbb T^d)}\leqslant 
2\pi_2^2\big[\|\hat u\|_{L^2(\mathbb T^d)}-(1,\hat u)\big];
$$
here we have used the fact that $\|\hat u\|_{L^2(\mathbb T^d)}=1$. Finally,
$$
|(1,\hat u)|\leqslant \big(1-\frac 1{2\pi_2^2}\big)\|\hat u\|_{L^2(\mathbb T^d)}.
$$
Letting $C_0=\big(1-\frac 1{2\pi_2^2}\big)$ from \eqref{hat_ineq} we deduce
$$
\|\hat u\|_{L^2(\mathbb T^d)}=\|\tilde u+(1,\hat u)\|_{L^2(\mathbb T^d)}\leqslant
\|\tilde u\|_{L^2(\mathbb T^d)}+|(1,\hat u)|\leqslant \|\tilde u\|_{L^2(\mathbb T^d)}+
C_0\|\hat u\|_{L^2(\mathbb T^d)}.
$$
Thus $\|\hat u\|_{L^2(\mathbb T^d)}\leqslant(1-C_0)^{-1}\|\tilde u\|_{L^2(\mathbb T^d)}=2\pi_2^2
\|\tilde u\|_{L^2(\mathbb T^d)}$, and \eqref{L4-5} holds with $C_3=(2\pi_2^2)^{-1}$.

\end{proof}

From  \eqref{L4-2} - \eqref{L4-5} it follows that
\begin{equation}\label{L4-6}
\int\limits_{\mathbb{T}^d} \partial_s (u^2 p^N) d \xi \le - C_4 \int\limits_{\mathbb{T}^d}  (u - C(N))^2 \, d \xi.
\end{equation}
Due to \eqref{prop1}--\eqref{prop2} we have
\begin{equation}\label{L4-7}
\int\limits_{\mathbb{T}^d} \partial_s (u^2 p^N) d \xi =   \int\limits_{\mathbb{T}^d} \partial_s \big((u - C(N))^2 \, p^N \big) \,d \xi.
\end{equation}
Taking into account this relation and the statement of Lemma \ref{L3} we can rewrite inequality \eqref{L4-6} 
in the following form:
\begin{equation}\label{L4-8}
\partial_s \int\limits_{\mathbb{T}^d}  (u - C(N))^2 \, p^N  d \xi \le - C_5 \int\limits_{\mathbb{T}^d}  (u - C(N))^2 \,  p^N \, d \xi
\end{equation}
with $C_5>0$.
If we denote $F(s) = \int\limits_{\mathbb{T}^d}  (u (\xi,s)- C(N))^2 \, p^N (\xi,s)  d \xi$,  then by \eqref{L4-8} $F(s)$ is a decreasing
non-negative function. 
By Theorem 3.1.1. in \cite{Laksh} inequality \eqref{L4-8} implies that
\begin{equation}\label{L4-9}
F(s) \le F(0) e^{- C_5 \, s} \ \ \hbox{for all } s \in [0,N].
\end{equation}
Considering the definition of $C(N)$ we also have  $F(0)\le \pi_2 \| u_0 \|^2_{L^2(\mathbb{T}^d)}$. 
Taking $s=N$ and using again the two-sided estimate \eqref{prop3} 
we obtain from \eqref{L4-9} the required estimate \eqref{prop4}.
Lemma \ref{L4} is proved.

\bigskip
\noindent
{\it Proof of Lemma \ref{L4-add}.} \\
In the remaining part of this section for brevity we denote $\|\cdot\|_{L^2(\mathbb T^d)}$ simply by $\|\cdot\|$. 
The sequence $\{ C(N) \}, N \to \infty,$ is a Cauchy sequence. Indeed, for any $K>0$ using \eqref{prop4} we get
\begin{equation}\label{L4-10}
|C(N) - C(N+K)|  \le \| C(N) - u(\cdot, N)\| + \| u(\cdot, N) - C(N+K)\| \le 2 C_4 e^{-\gamma_0 N} \|u_0\|.
\end{equation}
Therefore,  the sequence $\{C(N)\}$ converges, we denote its limit by $C_{\infty}$.
The inequality \eqref{prop4-add1} is the immediate consequence of \eqref{L4-10} and  \eqref{prop4-1}. The lemma is proved.

\bigskip
\noindent
{\it Proof of Lemma \ref{L4-bis}.} \\
Recalling that  $p^N$ and $p^{N+K}$ are solutions of the problems
$$
\left\{\begin{array}{c}
-\partial_s p^N=A^*(s)p^N\\[2mm]
p^N(\xi,N)=1
\end{array}\right.
\quad
\hbox{ and }
\quad
\left\{\begin{array}{c}
-\partial_s p^{N+K}=A^*(s)p^{N+K}\\[2mm]
p^N(\xi,N+K)=1,
\end{array}
\right.
$$
we conclude that  $q^N(\xi,s)= p^N (\xi,s) - p^{N+K} (\xi,s), \ s \in [0,N]$, 
satisfies the problem
$$
-\partial_s q^N=A^*(s)q^N,\qquad   
q^N(\xi,N)=q_0(\xi),
$$
with $q_0(\xi) = 1 - p^{N+K}(\xi, N)$.
In follows from Lemma \ref{L1} and \ref{L2} that
\begin{equation}\label{L4-bis-1}
\int\limits_{\mathbb{T}^d} q^N (\xi,s) d \xi = 0, \quad \int\limits_{\mathbb{T}^d} u(\xi,s)  q^N(\xi, s)  d \xi= C(N) - C(N+K)
\end{equation}
for all $ s \in [0,N]$.

Using \eqref{prop4-add1},\eqref{prop3}  and \eqref{L4-bis-1}  we get
$$
 \Big|  \int\limits_{\mathbb{T}^d} u_0(\xi)  q^N(\xi, 0)  d \xi \Big| =  \Big|  \int\limits_{\mathbb{T}^d} u(\xi,N)  q^N(\xi, N)  d \xi \Big|
$$
$$
= \Big|  \int\limits_{\mathbb{T}^d} (u(\xi,N) - C_\infty)  q^N(\xi, N)  d \xi + C_\infty  \int\limits_{\mathbb{T}^d} q^N(\xi, n)  d \xi  \Big|
$$
$$
\le \|u(\cdot, N) - C_{\infty}\| \, \|q^N(\cdot, N)\| \le C_5 \, e^{-\gamma_0 N} \|u_0\| \, \|q_0\| \le C_6 \, e^{-\gamma_0 N} \|u_0\|.
$$
Thus we obtain the following estimate
\begin{equation}\label{L4-bis-2}
 |(u_0, q^N(\cdot,0))| = \Big|  \int\limits_{\mathbb{T}^d} u_0(\xi)  q^N(\xi, 0)  d \xi \Big| \le  C_6 \, e^{-\gamma_0 N} \|u_0\|,
\end{equation}
that is valid for any $u_0 \in L^2(\mathbb{T}^d)$. Consequently,
\begin{equation}\label{L4-bis-3}
\| q^N(\cdot, 0)\|= \sup_{\|u_0\|=1} |(u_0, q^N(0))| \le  C_6 \, e^{-\gamma_0 N}.
\end{equation}
In the same way, 
for any $ s \in [0,N)$, 
we obtian
\begin{equation}\label{L4-bis-3}
\| q^N(\cdot, s)\|= \sup_{\|u\|=1} |(u, q^N(s))| \le  C_6 \, e^{-\gamma_0 (N-s)}.
\end{equation}
This completes the proof of Lemma \ref{L4-bis}.

\bigskip
\noindent
{\it Proof of Proposition \ref{uniq}.} \\
Our goal is to show that the limit
\begin{equation}\label{uniq-P1}
p_\infty(\xi, s) = \lim\limits_{N \to \infty} p^N(\xi, s)
\end{equation}
exists and  $p_\infty(\xi, s)$ is a stationary solution of the equation $-\partial_s  p = A^\ast_\omega(s)p$.
According to estimate \eqref{prop4-bis}, for any $s\in\mathbb R$  $\{ p^N(\cdot,s) \}$ is a Cauchy sequence 
in $L^2(\mathbb T^d)$, as $N \to \infty$. Therefore, the limit \eqref{uniq-P1} exists for any $s \in \mathbb{R}$:
\begin{equation}\label{lim_pinf}
\|p^N(\cdot,s) - p_\infty(\cdot, s)\| \to 0 \;\;\mbox{as } \; N \to \infty,
\end{equation}
and we can show in the same way as above that
$$
\|p^N(\cdot,s) - p_\infty(\cdot, s)\| \le C_7 \, e^{- \gamma_0 (N-s)}, \quad s<N.
$$
Moreover, the convergence is uniform on any bounded interval $[-L,L]$, i.e.
$$
\lim_{N \to \infty} \, \sup_{s \in [-L,L]} \|p^N(\cdot,s) - p_\infty(\cdot, s)\|\to 0.
$$
Due to \eqref{lim_pinf} the function $p_\infty$ inherits the bounds and normalization condition given in 
\eqref{prop2}-\eqref{prop3}. This yields \eqref{uniq-2}.

Using \eqref{uniq-2} and  the boundedness of operator $A^\ast$ we conclude that $p_\infty (\xi,s) \in C(\mathbb{R}, \, L^2(\mathbb{T}^d))$. Indeed,
$$
\|p_\infty(\xi, t) - p_\infty (\xi, s) \| = \| \int_{s}^{t} \partial_\tau p_\infty (\xi, \tau) d \tau  \| =  \| \int_{s}^{t} A^\ast(\tau) p_\infty (\xi, \tau) d \tau \| \le C_8 \, |t-s|.
$$

Since $\mu(\xi, \eta; s)$ is a stationary field, then 
$p^{N+s}(\xi,s)$ is also a stationary field for any $N>0$.  Since $p_\infty(\cdot,s)=\lim\limits_{N\to\infty}
p^{N+s}(\cdot, s)$,  we conclude that $ p_\infty(\xi,s)$ is a stationary field. Thus,
$$
p_\infty(\xi, s, \omega) = p_\infty(\xi, 0, \tau_s \omega).
$$

To prove the uniqueness of $ p_\infty(\xi,s,\omega)$ (up to a multiplicative constant) we assume that there exists another stationary solution $q (\xi,s,\omega)$ to the problem $-\partial_t  q = A^\ast_\omega(s)q$.
By the same arguments as in the proof of Lemma \ref{L2} we have
\begin{equation}\label{uniq-P2}
\int_{\mathbb{T}^d}q (\xi,s,\omega) d\xi = (q (\cdot, s,\omega), \, 1) = const.
\end{equation}
We denote this constant by $\alpha$. Assume first that $\alpha =0$. Then 
following the line of the proof of Lemma \ref{L4-bis} we obtain
%
\begin{equation*}\label{L4-bis-33}
\| q(\cdot, 0,\omega)\| \leqslant  C_6 \, e^{-\gamma_0 N}\, \|q(\cdot, N,\omega)\|
\end{equation*}
for all $N>0$.
Taking the average yields the following estimate:
\begin{equation}\label{exp_ubyv}
\mathbb{E} \| q(\cdot, 0, \omega)\| \le  C_6 \, e^{-\gamma_0  N} \ \mathbb{E} \| q(\cdot, N, \omega)\|
\end{equation}
  Due to the stationatity of $q(\cdot,s)$ we have
  $\mathbb{E} \| q(\cdot, 0, \omega)\| =   \mathbb{E} \| q(\cdot, N, \omega)\| $. Combining this relation with \eqref{exp_ubyv} we conclude that $ \mathbb{E} \| q(\cdot, s, \omega)\| = 0$, and thus $q(\xi,s,\omega) \equiv 0$ a.s.

If $\alpha \neq 0$, 
then $\varrho = \alpha p_\infty - q $ is a stationary solution of the equation $-\partial_t  q = A^\ast_\omega(s)q$, such that the integral for $\varrho(\cdot, s)$  over $\mathbb T^d$ 
is equal to 0. Thus, $\varrho \equiv 0$ and consequently, $q = {\alpha} p_\infty$.

This completes the proof of Proposition \ref{uniq}.

\bigskip
\noindent
{\it Proof of Lemma \ref{Du}.} \\
Using Duhamel's principle for inhomogeneous linear evolution equation \eqref{Du-1} we obtain the following formal representation for its solution: 
\begin{equation}\label{DuP-1}
v(\xi,t) =  \int\limits_{-\infty}^t U(t, \tau) g_\omega(\xi,\tau) \, d \tau, 
\end{equation}
here $U(t,s)$ is the linear operator that maps the initial condition $g_\omega(\xi,s)$ 
in the problem 
$$
\partial_\tau u=A_\omega(\tau)u,\qquad u(\xi,s)= g_\omega(\xi,s),
$$ 
to the solution $u(\xi,t)$ of this problem evaluated at time $t$. This operator is called evolution operator.
Our goal is to show that the integral on the right-hand side of \eqref{DuP-1} converges and that indeed the function
defined by this integral is
a solution of equation  \eqref{Du-1}. The stationarity of this function is clear.
Taking into account \eqref{Du-3},
by Lemma \ref{L4-add} (see \eqref{prop4-add1}) we obtain
\begin{equation}\label{DuP-3}
  \| U(t, \tau) g_\omega(\cdot,\tau)\|_{L^2(\mathbb{T}^d)} \le 
  C_2 e^{-C_3(t-\tau)}\|g_\omega(\cdot,\tau)\|_{L^2(\mathbb T^d)}.
\end{equation}
This implies convergence of the integral in \eqref{DuP-1}.  For an arbitrary $s\in\mathbb R$ the function 
$v_s(\xi,t) =  \int\limits_{s}^t U(t, \tau) g_\omega(\xi,\tau) \, d \tau$ is a solution of the problem
$\partial_t u(\xi,t)=A(t)u(\xi,t)+ g(\xi,t)$ for $t>s$,\  $u(\xi,s)=0$. Passing to the limit $s\to\ -\infty$
we conclude that $v(\xi,t)$ defined in formula \eqref{DuP-1} is a solution of equation \eqref{Du-1}.


\end{document}